\documentclass[11pt]{article}

\usepackage{MA_Titlepage}
\usepackage[margin=1.0in]{geometry}
\usepackage[onehalfspacing]{setspace}
\usepackage[english]{babel}
\usepackage[utf8x]{inputenc}
\usepackage{amsmath}
\usepackage{amssymb}
\usepackage{tikz-cd}
\usepackage{amsthm}
\usepackage{graphicx}
\usepackage[colorinlistoftodos]{todonotes}
\usepackage{hyperref}
\usepackage{mathtools}
\usepackage{mathrsfs}
\usepackage{dsfont}
\usepackage{diagbox}
\usepackage[nottoc,numbib]{tocbibind}

\hypersetup{
	linktocpage
}

\theoremstyle{definition}
\newtheorem{definition}{Definition}[section]
\newtheorem{remark}[definition]{Remark}
\theoremstyle{plain}
\newtheorem{lemma}[definition]{Lemma}
\newtheorem{corollary}[definition]{Corollary}
\newtheorem{theorem}[definition]{Theorem}
\newtheorem{proposition}[definition]{Proposition}
\newtheorem*{proposition*}{Proposition}
\theoremstyle{remark}

\theoremstyle{definition}
\theoremstyle{definition}

\newcommand{\Spa}[1]{\mathrm{Spa}(#1, {#1}^{+})}
\newcommand{\spa}[1]{\mathrm{Spa}\,#1}
\newcommand{\spec}[1]{\mathrm{Spec}\,#1}
\newcommand{\pair}[1]{(#1, {#1}^{+})}
\newcommand{\Tate}[2]{#1\langle X_{1},\dots, X_{#2} \rangle}

\newcommand{\ph}{\phantom{a}}

\authornew{Dmytro Rudenko}
\date{\today}

\betreuer{Advisor: Dr. Ian Gleason}
\zweitgutachter{Second Advisor: Prof. Dr. Peter Scholze}
\institut{Mathematisches Institut}
\title{On \'etale maps of sous-perfectoid adic spaces}
\ausarbeitungstyp{Master's Thesis  Mathematics}

\begin{document}
	
	\maketitle

	\begin{abstract}
		We prove a proposition for \'etale maps of sous-perfectoid adic spaces analogous to \cite[Proposition 1.7.1]{Hu96}, where the latter can be viewed as a form of Jacobian criterion for \'etale maps of noetherian\footnote{We point out that every ring considered in \cite{Hu96} satisfies some noetherian conditions, see \cite[Assumption 1.1.1]{Hu96}, while sous-perfectoid spaces are not always noetherian.} adic spaces. This allows us to provide some technical details to an important claim from the theory of \'etale maps of perfectoid spaces. Namely, we show how our proposition implies \cite[Proposition 6.4 (iv)]{S17} -- a sort of noetherian approximation for perfectoid rings. Afterwards, we give an explicit local description of the sheaf of differentials associated to a smooth map of sous-perfectoid adic spaces, as defined by Fargues-Scholze in \cite{FS}, in terms of the module of differentials. 
	\end{abstract}
	
	\tableofcontents

	\section{Introduction}
	
	The propositions mentioned in the abstract are the analogues of the following proposition.
	
	\begin{proposition}\label{proposition 0}
		Let $g\colon\spec{B}\to\spec{A}$ be a map of affine schemes. Then $g$ is \'etale if and only if there exists a presentation $B\cong A[X_{1},\dots,X_{n}]/(f_{1},\dots, f_{n})$ such that the determinant of the associated Jacobian matrix $\mathrm{det}(\frac{\partial f_{i}}{\partial X_{j}})_{1\leqslant i, j\leqslant n}$ is a unit in $B$.
	\end{proposition}
	
	This is a global and geometric description of \'etale $A$-algebras as global complete intersections of $n$ equations in $n$ variables. Adic spaces generalize schemes, (formal schemes, and rigid analytic varieties,) and therefore it would be nice to extend the above proposition for adic spaces. As was shown by Huber in \cite{Hu96}, the theory of adic spaces extends the one of schemes (and formal schemes, and rigid analytic varieties) quite rigidly under noetherian assumptions. In particular, the original formulation of Proposition 1.7.1 from \cite{Hu96} is a direct generalization of Proposition \ref{proposition 0}. Here we only include its simplified formulation.
	
	\begin{proposition}[simplified {\cite[Proposition 1.7.1]{Hu96}}]\label{proposition 1.7.1}
		Let $\pair{A}, \pair{B}$ be Huber pairs with $A, B$ complete Tate and strongly noetherian. Let $g\colon\Spa{B}\to\Spa{A}$ be a map of affinoid adic spaces. Then $g$ is \'etale if and only if there exists a presentation \[ \pair{B}\cong \pair{\Tate{A}{n}/(f_{1},\dots,f_{n})} \] such that the determinant of the associated Jacobian matrix $\mathrm{det}( \frac{\partial f_{i}}{\partial X_{j}})_{1\leqslant i, j\leqslant n}$ is a unit in $B$.
	\end{proposition}
	
	While (the original formulation of) Proposition \ref{proposition 1.7.1} yields a generalization of Proposition \ref{proposition 0} to locally noetherian adic spaces, it does not cover all cases of interest, because perfectoid adic spaces are usually not locally noetherian. The main proposition of this thesis is that, with a suitable definition\footnote{lower in the introduction} of \'etale maps of perfectoid spaces, Proposition \ref{proposition 1.7.1} remains true for those.
	
	\begin{proposition}[main proposition]\label{upgraded main proposition}
		Let $g\colon\Spa{B}\to\Spa{A}$ be a map of affinoid perfectoid spaces. Then $g$ is \'etale if and only if there exists a presentation \[ \pair{B}\cong \pair{\Tate{A}{n}/(f_{1},\dots,f_{n})} \] such that the determinant of the associated Jacobian matrix $\mathrm{det}( \frac{\partial f_{i}}{\partial X_{j}})_{1\leqslant i, j\leqslant n}$ is a unit in $B$.
	\end{proposition}
	
	One application of this proposition is the following version of noetherian approximation for perfectoid rings from \cite{S17}.
	
	\begin{proposition}[simplified {\cite[Proposition 6.4 (iv)]{S17}}]\label{noetherian approximation}
		Let $X_{i} = \Spa{R_{i}}$, $i\in I$ be a cofiltered system of affinoid perfectoid spaces for some small index category $I$. Let $X = \Spa{R}$ be the inverse limit, where $R^{+} $ is the $\varpi$-adic completion of $\varinjlim_{i} R_{i}^{+}$, and $R = R^{+}[\frac{1}{\varpi}]$. Here, $\varpi\in R_{i}^{+}$ denotes any compatible choice of pseudouniformizers for large $i$. Then for every perfectoid affinoid space $\Spa{B}$ along with an \'etale map $g\colon\Spa{B}\to\Spa{R}$ there exists $i\in I$, a perfectoid affinoid space $\Spa{B_{i}}$, an \'etale map $g_{i}\colon\Spa{B_{i}}\to\Spa{R_{i}}$, and a cartesian square
		\[ \begin{tikzcd}
			\Spa{B}\arrow[d, "g"]\arrow[r]&\Spa{B_{i}}\arrow[d, "g_{i}"]\\
			\Spa{R}\arrow[r, "\pi_{i}"]&\Spa{R_{i}},\\
		\end{tikzcd} \]
		where $\pi_{i}$ is the canonical $i$-th projection from the inverse limit.
	\end{proposition}
	
	We explain how the above proposition follows from Proposition \ref{upgraded main proposition} in section \ref{section proof of noetherian approximation}.
	
	We now touch the main points of the proof of Proposition \ref{upgraded main proposition} while simultaneously explaining the structure of this thesis.
		
	The proof of the Proposition \ref{upgraded main proposition} is contained in section \ref{section proof of the main proposition}. While it is similar to the proofs of propositions \ref{proposition 0}, \ref{proposition 1.7.1}, we need to account for technical difficulties which do not occur under noetherian assumptions. 
	
	First of all, all proofs require some notion of a smooth map of adic spaces and of the associated sheaf of differentials. The notion of formal smoothness was first defined in \cite[Ch. 0, Def. 19.3.1]{EGA IV} as a property of a map of topological rings to admit lifts of thickenings, to which we refer as to the property of ``lifting of differentials". This was done with view towards formal schemes and so, using the same criterion of lifting of differentials, one can define formally smooth maps of formal schemes as well as the associated sheaves of differentials. In fact, as was shown by Huber in \cite{Hu96}, an analogous definition of a smooth map of adic spaces via lifting of differentials works well in the category of noetherian adic spaces. However, the usual definition of a smooth map would not make sense in the category of perfectoid adic spaces, because every perfectoid ring is reduced. Thus, the first technical problem is to define the notion of a smooth map of perfectiod spaces and of the associated sheaf of differentials. This problem was already solved in \cite{FS}. The authors of \cite{FS} adopt the approach of working with sous-perfectoid rings as defined in \cite[Definition 7.1]{HK} instead of perfectoid rings. Even though still reduced, sous-perfectoid rings exhibit better stability properties, for example, if $R$ is sous-perfectoid, then so is $\Tate{R}{n}$. Second of all, motivated by Huber's \cite[Corollary 1.6.10]{Hu96} and by the uniformizing parameters from the world of schemes, Fargues and Scholze define smoothness of a map of sous-perfectoid adic spaces as follows. 
	\begin{definition}[{\cite[Definition IV.4.8]{FS}}]\label{smooth map}
		\begin{enumerate}
			\item[iii)] A map $f\colon Y\to X$ of sous-perfectoid adic spaces is called \emph{smooth} if one can cover $Y$ by open subsets $V\subset Y$, for each of which there exist an integer $d\geq 0$ and an \'etale map $V\to\mathbb{B}_{X}^{d}$ over $X$.
			\item[ii)] A map $f\colon Y\to X$ of sous-perfectoid adic spaces is called \emph{\'etale} if locally on the source and target it can be written as an open immersion of adic spaces followed by a finite  \'etale map.
			\item[i)] A map $f\colon Y\to X$ of sous-perfectoid adic spaces is called \emph{finite \'etale} if for all open affinoid $\spa{A}\subset X$ the pullback $Y\times_{X}\spa{A}\subset Y $ is also open affinoid $\spa{B}$, where $B$ is a finite \'etale $A$-algebra and $B^{+}$ is the integral closure of the image of $A^{+}$.
		\end{enumerate}
	\end{definition}
	Now, with that definition, a statement analogous to uniformizing parameters holds for smooth maps of sous-perfectoid adic spaces, by definition. Furthermore, for every smooth map $f\colon Y\to X$ of sous-perfectoid adic spaces Fargues and Scholze define the sheaf of differentials $\Omega^{1}_{Y/X}$.
	\begin{definition}[{\cite[Definition IV.4.11]{FS}}]\label{definition sheaf of differentials}
		Let $f\colon Y\to X$ be a smooth map of sous-perfectoid adic spaces, with diagonal $\Delta_{f}\colon Y\to Y\times_{X}Y$. Let $\mathcal{I}_{Y/X}\subset \mathcal{O}_{Y\times_{X}Y}$ be the ideal sheaf. Then \[ \Omega^{1}_{Y/X}:=\mathcal{I}_{Y/X}/\mathcal{I}_{Y/X}^{2} \] considered as a $\mathcal{O}_{Y\times_{X}Y}/\mathcal{I}_{Y/X} = \mathcal{O}_{Y}$-module.
	\end{definition}
	As was verified in \cite{FS}, the sheaf of differentials and smooth maps defined in \emph{loc. cit.} have properties similar to those of the usual sheaf of differentials and usual smooth maps (defined via lifting of differentials in noetherian setting).
	
	Now we are able to point out the two main technical difficulties from the proof of Proposition \ref{upgraded main proposition}. In fact, in order to prove Proposition \ref{upgraded main proposition}, we first prove the following (a priori weaker) statement:
	\begin{proposition}\label{main proposition}
		Let $g\colon\Spa{B}\to\Spa{A}$ be an \'etale map of affinoid sous-perfectoid spaces. Then it factors through a global closed immersion\footnote{for the definition of this see Definition \ref{global closed immersion}.} $f$
		\[\begin{tikzcd}
			\ph&\Spa{\Tate{A}{n}}\arrow[d, "\pi"]\\
			\Spa{B}\arrow[r, "g"]\arrow[ur, "f"]&\Spa{A}
		\end{tikzcd}\]
		such that the natural map from the conormal exact sequence $\mathcal{I}/\mathcal{I}^{2}\to f^{*}\Omega^{1}_{\pi}$ is an isomorphism, where $\mathcal{I}$ is the ideal sheaf of the closed immersion $f$.
	\end{proposition}
	
	We would like to emphasize that in the world of sous-perfectoid adic spaces, even for affinoids, one only has a local on the target and sourse definition of an \'etale map (the definition \ref{smooth map} ii) above). It is a pervading theme in this setting that statements of the form ``local implies global" are difficult to prove (see remark \ref{local implies global remark} for more discussion). It is therefore essential that the closed immersion $f$ from Proposition \ref{main proposition} is global. Therefore, the first technical difficulty (maybe surprisingly) is to prove that the global closed immersion $f$ even exists. Indeed, the way Huber proves ``local implies global" under noetherian hypotheses is different from the way we do it (without noetherian hypotheses but for sous-perfectoid spaces). We prove Proposition \ref{main proposition} in subsection \ref{subsection global closed immersion}.
	
	The second technical difficulty is to get Proposition \ref{upgraded main proposition} from Proposition \ref{main proposition}.  To be able to find a presentation of $B$ we need to be able to say that the ideal sheaves $\mathcal{I}$ and $\mathcal{I}/\mathcal{I}^{2}$ are determined by their global sections $I$ and $I/I^{2}$. For this, we need a theory of ``quasicoherent" sheaves on adic spaces, and the theory of pseudocoherent sheaves of Kedlaya-Liu from \cite{KL16} saves our day. We have to carefully check that we never leave the world of pseudocoherent sheaves, and then we can upgrade Proposition \ref{main proposition} to Proposition \ref{upgraded main proposition}, which we do in subsection \ref{subsection global complete intersection}.
	
	Finally, we compute locally the sheaf of differentials from Definition \ref{definition sheaf of differentials} in terms of the module of differentials. We prove the following proposition in section \ref{section the sheaf of differentials}.
	
	\begin{proposition}\label{description of the sheaf of differentials}
		Let $f\colon Y\to X$ be a smooth map of sous-perfectoid adic spaces. For any open affinoids $\spa{B}\subset Y$ and $\spa{A}\subset X$ with $f(\spa{B})\subset \spa{A}$ holds \[ \Omega^{1}_{Y/X}\mid_{\spa{B}}\cong \Omega^{1}_{\spa{B}/\spa{A}}\cong \widehat{\Omega^{1}_{B/A}}\footnote{rather, its tildification $\widetilde{\widehat{\Omega^{1}_{B/A}}}$, which we decided to omit for notational purposes.}, \]
		where $\widehat{\Omega^{1}_{B/A}}$ is the usual module of differentials completed w.r.t. a suitable topology (defined later).
	\end{proposition}
	
	\begin{center}
		\textbf{Acknowledgements}
	\end{center}
	
	I would like to express my sincere gratitude to my master's thesis supervisor Ian Gleason for regularly supporting me, listening through all of my mathematical endeavours, and encouraging me. Without his continuous guidance and proofreadings of drafts, this thesis would have had much less content and been much less readable. I am grateful to Peter Scholze for giving me this thesis problem, and for directing me in some critical moments throughtout my thesis. I am thankful to Max Hauck and Guido Bosco for our numerous and valuable discussions.

	\section{Preliminaries}\label{section preliminaries}
	
	In this section we remind the reader some of the useful facts about Huber pairs and rational localizations, as well as define something that we call a ``LOCT closed immersion", which we need for the proof of Proposition \ref{main proposition}. Since this thesis is not self-contained, the reader is assumed to be familiar with adic spaces. Nevertheless, the author had already written some preliminary facts/definitions and decided not to delete them. Experienced reader can skip subsection \ref{subsection rational localization of a quotient map is a quotient map}.
	 
	\subsection{Rational localization of a quotient map is a quotient map}\label{subsection rational localization of a quotient map is a quotient map}
	
	We start by stating a couple of lemmas. The first lemma is to remind the reader that a continuous group homomorphism of topological groups is open and surjective if and only if it is a topological quotient map.
	
	\begin{lemma}\label{group action quotient}
		Let a group $G$ act on a topological space $X$ by homeomorphisms. If we endow $X/G$ with the quotient topology with respect to the projection $\pi\colon X\to X/G$, then the projection $\pi$ is an open map.
	\end{lemma}
	\begin{proof}
		This is clear.
	\end{proof}
	
	\begin{lemma}\label{image of the closure is the closure of the image}
		Let $f\colon X\to Y$ be a map of topological spaces, $S\subset X$ be a subset. If $\overline{S}$ denotes the closure, we have $\overline{f(\overline{S})} = \overline{f(S)}$. In particular, $f(\overline{S})\subset \overline{f(S)}$.
	\end{lemma}
	\begin{proof}
		This is clear.
	\end{proof}
		
	The rest of this subsection is devoted to Huber rings and pairs. Given a topological ring $A$, the set of power-bouned elements of $A$ is denoted by $A^{\circ}$.
	
	\begin{definition}
		A topological ring $A$ is a \emph{Huber ring} if it contains an open subring $A_{0}\subset A$ such that the induced topology on $A_{0}$ is $I$-adic for some finitely generated ideal $I$ of $A_{0}$. Any such pair $(A_{0}, I)$ is called a \emph{pair of definition} of the Huber ring $A$. 
	\end{definition}
	
	\begin{definition}
		A Huber ring $A$ is called \emph{Tate} if it contains a topologically nilpotent unit $\varpi\in A^{\times}$, called \emph{pseudouniformizer}.
	\end{definition}
	
	From this point onwards, all Huber rings are assumed to be Hausdorff complete and Tate, unless otherwise explicitly stated. 
		
	\begin{definition}
		A \emph{Huber pair} is a pair $\pair{A}$, where $A$ is a Huber ring and $A^{+}\subset A$ is an open, integrally closed subring of $A$ contained in $A^{\circ}$. A \emph{map of Huber pairs } $f\colon \pair{A}\to \pair{B}$ is a continuous ring homomorphism $f\colon A\to B$ which satisfies $f(A^{+})\subset B^{+}$.
	\end{definition}

	\begin{definition}
		A map of Huber pairs $f\colon\pair{A}\to\pair{B}$ is called a \emph{quotient map} if $f$ is open surjective and $B^{+} = f(A^{+})^{int}$ is the integral closure of the image of $A^{+}$.
	\end{definition}

	Every quotient map of Huber pairs is of the following form and has the following universal property.
	
	\begin{lemma}\label{quotient pair}
		Let $\pair{A}$ be a Huber pair and let $J\subset A$ be a closed ideal. Then the ring $A/J$ endowed with the quotient topology is a Huber ring, and if we define $A/J^{+}:= (A^{+}/(A^{+}\cap J))^{int}$ to be the integral closure of the image of $A^{+}$ in $A/J$, the corresponding pair $\pair{A/J}$ is a Huber pair. The canonical quotient map $\pair{A}\to\pair{A/J}$ satisfies the following universal property: a map of Huber pairs $f\colon \pair{A}\to\pair{B}$ factors uniquely through $\pair{A/J}$ if and only if $f(J) = \left\lbrace 0\right\rbrace $.
	\end{lemma}
	\begin{proof}
		So far omitted, but this is clear.
	\end{proof}
	
	We now define rational localizations of Huber pairs. An alternative definition can be found inside Proposition \ref{rational localization}.
	\begin{definition}
		Let $\pair{A}$ be a Huber pair, $T = \left\lbrace t_{1}, \dots, t_{n}\right\rbrace \subset A$ be a finite subset of elements that generate the unit ideal $(t_{1},\dots, t_{n}) = A$, and $s\in A$ be any element. The \emph{rational localization} of $\pair{A}$ with respect to $T$ and $ s$ is a Huber pair $\pair{A\left\langle\frac{T}{s}\right\rangle}$ which is defined as follows. Choose a pair of definition $(A_{0}, I)$ of the Huber ring $A$. Endow the ring $A[\frac{1}{s}]$ with a topology by declaring $(A_{0}[\frac{t_{1}}{s},\dots, \frac{t_{n}}{s}], I\cdot A_{0}[\frac{t_{1}}{s},\dots, \frac{t_{n}}{s}])$\footnote{the ring $A_{0}[\frac{t_{1}}{s},\dots, \frac{t_{n}}{s}]$ is formed by adjoining the elements $\frac{t_{1}}{s},\dots, \frac{t_{n}}{s}$ to the image of $A_{0}$ under the localization map $A\to A[\frac{1}{s}]$.} to be a pair of definition of $A[\frac{1}{s}]$. The resulting  Huber ring, which is not complete in general, is denoted by $A( \frac{T}{s}) $. Defining $ A( \frac{T}{s})^{+} := $ $(A^{+}[\frac{t_{1}}{s},\dots, \frac{t_{n}}{s}])^{int}$ as the integral closure of the ring generated by the image of $A^{+}$ and $\frac{t_{1}}{s},\dots, \frac{t_{n}}{s}$ yields the (non-complete) Huber pair $\pair{A( \frac{T}{s})}$. The completion of this (non-complete) Huber pair is a Huber pair denoted by $\pair{A\left\langle \frac{T}{s}\right\rangle}$. 
	\end{definition}
	
	The remark below can be found in \cite{Hu94} Statement 1.2 just above the Proposition 1.3 or in \cite{Wed} Proposition/Definition 5.51 along with the beginning of Section 8.1.
	
	\begin{remark}
		The canonical map of Huber pairs $\pair{A}\to\pair{A\left\langle \frac{T}{s}\right\rangle}$ satisfies the following universal property. A map of Huber pairs $f\colon\pair{A}\to \pair{B}$ factors uniquely through $\pair{A\left\langle \frac{T}{s}\right\rangle}$ if and only if $f(s)\in B$ is invertible and for all $i = 1,\dots, n$ holds $\frac{f(t_{i})}{f(s)}\in B^{+}$.
	\end{remark}
		
	We mention an alternative way of defining a rational localizaton of a Huber pair, and for the convenience of the reader we prove the equivalence of the constructions in the proposition below.
	
	\begin{proposition}[cf. {\cite[Lemma 2.4.13]{KL15}}]\label{rational localization}
		Let $\pair{A}$ be a Huber pair, $T = \left\lbrace t_{1}, \dots, t_{n}\right\rbrace \subset A$ be a finite subset of elements that generate the unit ideal $(t_{1},\dots, t_{n}) = A$, and $s\in A$ be any element. The quotient of the Huber pair $\pair{\Tate{A}{n}}$ by the closure of the ideal $(t_{1}-sX_{1},\dots, t_{n}-sX_{n})$ is a Huber pair isomorphic to $\pair{A\left\langle \frac{T}{s}\right\rangle}$, and this isomorphism is unique.
	\end{proposition}
	\begin{proof}
		It is sufficient to show that this quotient satisfies the same universal property as $\pair{A\left\langle \frac{T}{s}\right\rangle}$. For this, consider an arbitrary map of Huber pairs $f\colon\pair{A}\to \pair{B}$. Assume it factors through \[ \pair{\Tate{A}{n}/\overline{(t_{1}-sX_{1},\dots, t_{n}-sX_{n})}} .\]
		By lemma \ref{quotient pair}, it is equivalent to saying that the corresponding map \[ g\colon\pair{\Tate{A}{n}}\to \pair{B} \] sends $\overline{(t_{1}-sX_{1},\dots, t_{n}-sX_{n})}$ to $\left\lbrace 0\right\rbrace $. By lemma \ref{image of the closure is the closure of the image} and by the assumption that $B$ is Hausdorff, it is equivalent to saying that $(t_{1}-sX_{1},\dots, t_{n}-sX_{n})$ is sent to $\left\lbrace 0\right\rbrace $. Explicitly, it means that for all $i = 1,\dots, n$ holds \[ g(t_{i}-sX_{i}) = g(t_{i})-g(s)g(X_{i}) = f(t_{i})-f(s)g(X_{i}) = 0.\]
		Therefore, such a map $g$ is equivalent to choosing $g(X_{1}),\dots, g(X_{n}) \in B^{+}$ such that for all $i = 1,\dots, n$ holds $f(t_{i})=f(s)g(X_{i})$. Notice that this implies that $f(s)\in B$ is invertible because $(t_{1},\dots, t_{n})$ generate the unit ideal in $A$: 
		\[ 1 = a_{1}t_{1}+\dots+a_{n}t_{n} \implies 1= f(a_{1})f(t_{1})+\dots+f(a_{n})f(t_{n}) = f(a_{1})f(s)g(X_{i})+\dots+f(a_{n})f(s)g(X_{n}).\]
		Therefore, the existence of such a map $g$ carries the following information: $f(s)\in B$ is invertible and for all $i = 1,\dots, n$ holds $\frac{f(t_{i})}{f(s)}\in B^{+}$. Since all implications were equivalences, we are done.
	\end{proof}
	
	We can therefore view rational localizations as quotiens as above.
	
	We decided to prove the following lemma by hand.
	
	\begin{lemma}\label{open surjectivity is preserved by Tate extensions}
		Let $f\colon A\to B$ be an open surjective map of Huber rings. Then the induced map \[ g\colon\Tate{A}{n}\to \Tate{B}{n}\]
		sending $X_{i}$ to $X_{i}$ for all $i = 1,\dots, n$ is open surjective.
	\end{lemma}
	\begin{proof}
		An element of $\Tate{A}{n}$ is given by a convergent sequence $(a_{n})_{n\in\mathbb{N}}\subset A$, $a_{n} \to 0$. One can check that continuity of $g$ implies the explicit formula \[ g((a_{n})_{n\in\mathbb{N}}) = (f(a_{n}))_{n\in \mathbb{N}},\]
		where $(f(a_{n}))_{n\in \mathbb{N}}\subset B$ is a convergent sequence. 
		
		We first show surjectivity.  Let $(b_{n})_{n\in\mathbb{N}}\subset B$, $b_{n}\to 0$ be an arbitrary element of $\Tate{B}{n}$. We will find $(a_{n})_{n\in\mathbb{N}}\subset A$, $a_{n} \to 0$ with $f(a_{n}) = b_{n}$. Let $(A_{0}, I)$ be a pair of definition of $A$. Let $I, I^{2}, I^{3},\dots$ be a basis of neighbourhoods of $0\in A$. We know by assumption that $f(I), f(I^{2}), f(I^{3}), \dots$ are open, and hence each contains almost all $(b_{n})_{n\in\mathbb{N}}$. Let $b_{n_{i}-1}$ be the last among $(b_{n})_{n\in\mathbb{N}}$ that is outside of $f(I^{i})$. In particular, $b_{n_{i}}, b_{n_{i}+1},\dots$ are inside $f(I^{i})$. Choose arbitrary $a_{n}^{\prime}$ with $f(a_{n}^{\prime}) = b_{n}$. Then $f^{-1}(f(I^{i}))$ contains $a_{n_{i}}^{\prime}, a_{n_{i}+1}^{\prime}, \dots$. At the same time, $f^{-1}(f(I^{i})) = I^{i}+\ker f$. It means that $a_{n_{i}}^{\prime} = a_{n_{i}}+e_{n_{i}}$, where $a_{n_{i}}\in I^{i}$ and $e_{n_{i}}\in\ker f$. Then $f(a_{n_{i}}^{\prime}) = f(a_{n_{i}})$. Clearly we have the convergence $(a_{n_{i}})_{i\in\mathbb{N}}\to 0$. Moreover, by our choice of $n_{i}$, for any $n\geqslant n_{i}$ we have $a_{n}^{\prime}\in I^{i}+\ker f$. So we choose any decompositions $a_{n}^{\prime} = a_{n}+e_{n}$ (for those not yet chosen) with $a_{n}\in I^{i}$ and $e_{n}\in\ker f$. This way, for any $n_{i} < n < n_{i+1}$ we would have $a_{n_{i}}\in I^{i}$, $a_{n}\in I^{i}$, $a_{n_{i+1}}\in I^{i+1}$. Hence, $(a_{n})_{n\in \mathbb{N}}\to 0$.
		
		It is left to show that $g$ is open. We know from our assumption that for all $i\in\mathrm{N}$ the image $f(I^{i})$ is open in $B$. It is then sufficient to show that for all $i \in\mathbb{N}$ holds $g(\Tate{I^{i}}{n}) = \Tate{f(I^{i})}{n}$. Clearly $g(\Tate{I^{i}}{n}) \subset \Tate{f(I^{i})}{n}$. We now show the converse inclusion. Consider any element of $\Tate{f(I^{i})}{n}$, which is just a convergent sequence $(b_{n})_{n\in\mathbb{N}}\subset f(I^{i})$, $b_{n}\to 0$. Because $g$ was proven to be surjective, choose any $(a_{n})_{n\in\mathbb{N}}\subset A$, $a_{n} \to 0$ with $f(a_{n}) = b_{n}$. As $a_{n} \to 0$, there exists $M\in\mathbb{N}$ such that for all $n> M$ holds $a_{n} \in I^{i}$. We are almost done, but we need to also have $a_{n} \in I^{i}$ for $n \leqslant M$. For this, just choose any $a_{n} \in I^{i}$ with $f(a_{n}) = b_{n}$ (this is possible because $b_{n}\in f(I^{i})$). Changing finitely many terms does not spoil convergence, so we are done.
		
	\end{proof}
	
	Finally, we prove that rational localization of a quotient map of Huber pairs is again a quotient map of Huber pairs.
	
	\begin{proposition}\label{open surjectivity is preserved by rational localizations}
		Let $f\colon A\to B$ be an open surjective map of Huber rings. Let $T = \left\lbrace t_{1}, \dots, t_{n}\right\rbrace \subset A$ be a finite subset of elements that generate the unit ideal $(t_{1},\dots, t_{n}) = A$, and $s\in A$ be any element. Then the induced map of Huber rings \[ h\colon A\left\langle \frac{t_{1},\dots, t_{n}}{s}\right\rangle \to B\left\langle \frac{f(t_{1}),\dots, f(t_{n})}{f(s)} \right\rangle \]
		is open surjective.
	\end{proposition}
	\begin{proof}
		Consider the commutative diagram
		\[ 
		\begin{tikzcd}
			{\Tate{A}{n}} \arrow[r, "g"]\arrow[d, "\pi_{A}"] & {\Tate{B}{n}}\arrow[d, "\pi_{B}"]\\
			A\left\langle \dfrac{t_{1},\dots, t_{n}}{s}\right\rangle \arrow[r, "h"] & B\left\langle \dfrac{f(t_{1}),\dots, f(t_{n})}{f(s)} \right\rangle,\\
		\end{tikzcd}
		 \]
		 where the projection $\pi_{A}$ is the quotient by the ideal $\overline{(t_{1}-sX_{1},\dots, t_{n}-sX_{n})}$, the projection $\pi_{B}$ is the quotient by the ideal $\overline{(f(t_{1})-f(s)X_{1},\dots, f(t_{n})-f(s)X_{n})}$. It is indeed a commutative diagram because of the universal properties, or one can apply lemma \ref{image of the closure is the closure of the image}. By lemma \ref{open surjectivity is preserved by Tate extensions}, the map $g$ open surjective. By lemma \ref{group action quotient}, both projections $\pi_{A}$ and $\pi_{B}$ are open surjective. Hence, $h$ is open surjective.
		 
		 If we started with a quotient map of Huber \emph{pairs} $f\colon \pair{A}\to \pair{B}$, in the commutative square above topmost arrow, as well as the vertical arrows would have been quotient maps of Huber pairs. It would follow that the bottom arrow is a quotient map of Huber \emph{pairs}.
	\end{proof}

	\subsection{LOCT closed immersions}
	
	Here we define LOCT (locally on the target) closed immersions and prove that they satisfy COMP, i.e., that a composition of LOCT closed immersions is again a LOCT closed immersion.
		
	\begin{definition}\label{global closed immersion}
		A morphism of affinoid preadic\footnote{here it just means that we do not require our Huber rings to be sheafy.} spaces $f\colon \spa{B} \to \spa{A}$ is called a \emph{global closed immersion} if the induced map of Huber pairs $\pair{A}\to\pair{B}$ is a quotient map of Huber pairs.
	\end{definition}
	
	\begin{definition}
		Let $\pair{A}$, $\pair{B}$ be Huber pairs. A morphism of affinoid preadic spaces $f\colon \spa{B} \to \spa{A}$ is called a \emph{LOCT closed immersion} if there exists a cover $\bigcup_{i\in I}U_{i} = \spa{A}$ of $\spa{A}$ by open affinoids $U_{i} \cong \spa{C_{i}}$ such that for all $i\in I$ the preimage $f^{-1}(U_{i})\footnote{we really mean pullback but abuse notation. For the preimage of topological spaces we will use $|f|^{-1}.$} \cong \spa{D_{i}}$ is also open affinoid and the corresponding map of Huber pairs $\pair{C_{i}}\to\pair{D_{i}}$ is a quotient map.
	\end{definition}
	
	\begin{remark}
		It follows from the definition that the property of a map being LOCT closed immersion is local on the target (i.e., it satisfies LOCT).
	\end{remark}
	
	\begin{lemma}\label{quotient map of Huber pairs induces a topological closed immerion of spectra}
		Let $f\colon\pair{A}\to\pair{B}$ be a quotient map of Huber pairs, i.e., $\pair{B} \cong \pair{A/I}$ for some closed ideal $I\subset A$. Then the corresponding map on spectra $ \spa{A/I}\to \spa{A}$ is a topological closed immersion, with image $\left\lbrace v\in\spa{A}\mid \ker v \supset I\right\rbrace $.
	\end{lemma}
	\begin{proof}
		This is clear.
	\end{proof}
	
	\begin{lemma}\label{preimages in topological closed immerions}
		Let $\pair{A}$, $\pair{B}$ be Huber pairs and let $f\colon \spa{B} \to \spa{A}$ be a LOCT closed immersion. Then the underlying map of topological spaces $|f|\colon |\spa{B}|\to|\spa{A}|$ is a topological closed immersion, and for every point $x\in|\spa{B}|$ with an open neighbouhood $x\in V \subset |\spa{B}|$ there exists an open neighbourhood $f(x)\in U \subset|\spa{A}|$ whose preimage $|f|^{-1}(U)$ is contained in $V$.
	\end{lemma}
	\begin{proof}
		Since the property of being a topological closed immersion satisfies LOCT, we deduce that $|f|$ is a topological closed immersion from Lemma \ref{quotient map of Huber pairs induces a topological closed immerion of spectra}. The rest of the statement is true for any topological closed immersion. Indeed, if $x\in V \subset |\spa{B}|$ is an open set, then its complement $V^{c}\subset |\spa{B}|$ is closed, as well as its image $|f|(V^{c})$, so we can take $U$ to be the complement of the image $U:= (|f|(V^{c}))^{c}$.
	\end{proof}
	
	\begin{proposition}\label{LOCT closed immersion satisfies COMP}
		The property of a morphism of affinoid preadic spaces of being a LOCT closed immersion satisfies COMP.
	\end{proposition}
	\begin{proof}
		Let $\pair{A}$, $\pair{B}, \pair{C}$ be Huber pairs. Let $f\colon \spa{A} \to \spa{B}$ and \newline$g\colon \spa{B} \to \spa{C}$ be LOCT closed immerisons. We will show that $g\circ f$ is LOCT closed immersion. For this, let $x\in \spa{A}$ be an arbitrary point, and let $f(x)\in\spa{D}\subset \spa{B}$ be an open affinoid whose preimage $f^{-1}(\spa{D}) = \spa{D^{\prime}}$ under $f$ is also open affinoid with the corresponding map of Huber pairs being a quotient map. Let $g(f(x))\in\spa{E}\subset \spa{C}$  be an open affinoid whose preimage $g^{-1}(\spa{E}) = \spa{E^{\prime}}$ under $g$ is also open affinoid with the corresponding map of Huber pairs being a quotient map. \newline Let $f(x)\in V\subset \spa{D}\cap\spa{E^{\prime}}$ be an open affinoid that is a rational localization of both $ \spa{D}$ and $\spa{E^{\prime}}$ (simultaneously). In particular, $g$ restricts to a morphism of affinoid preadic spaces $g|_{V}\colon V \to \spa{E}$. By Lemma \ref{preimages in topological closed immerions} we can find an open neighbourhood $g(f(x))\in U\subset \spa{E}$ with its preimage $g^{-1}(U)$ contained in $V$. Consequently, we can find a rational localization $W\subset U$ of $\spa{E}$, so that its preimage $g^{-1}(W)$ is a rational localization of $V$. Therefore, $g^{-1}(W)$ is again a rational localization of both $ \spa{D}$ and $\spa{E^{\prime}}$ (simultaneously). If we apply Lemma \ref{open surjectivity is preserved by rational localizations} to the morphism $g^{-1}(W)\to W$ (which is a rational localization of the morpihsm $\spa{E^{\prime}}\to\spa{E}$), we obtain that the corresponding map of Huber pairs is a quotient map. At the same time, the preimage $f^{-1}(g^{-1}(W)) = [g\circ f]^{-1}(W)$ of $g^{-1}(W)$ under $f$ is a rational localization of $\spa{D^{\prime}}$. If we apply Lemma \ref{open surjectivity is preserved by rational localizations} to the morphism $f^{-1}(g^{-1}(W))\to g^{-1}(W)$ (which is a rational localization of the morpihsm $\spa{D^{\prime}}\to\spa{D}$), we obtain that the corresponding map of Huber pairs is a quotient map. Therefore, the composition $[g\circ f]^{-1}(W) \to W$ corresponds to a quotient map of Huber pairs, and we are done.
	\end{proof}
	
	\section{Proof of the main proposition}\label{section proof of the main proposition}
	
	\subsection{Global closed immersion}\label{subsection global closed immersion}
	
	In this section we prove Proposition \ref{main proposition}. Let us recall the setup. Let $g\colon\spa{B}\to\spa{A}$ be an \'etale morphism of sous-perfectiod adic spaces. We want to show that there exists a closed immersion $f\colon\spa{B}\to\spa{\Tate{A}{M}}$ such that the diagram 
	\[ 
	\begin{tikzcd}
		& \spa{\Tate{A}{M}}\arrow[d, "\pi"]\\
		\spa{B}\arrow[ur, "f"]\arrow[r, "g"] & \spa{A}
	\end{tikzcd}
	 \]
	commutes and such that the natural map $\mathcal{I}/\mathcal{I}^{2}\to f^{*}\Omega_{\mathbb{B}^{d}_{A}/\Spa{A}}$ is an isomorphism, where $\mathcal{I}$ is the ideal sheaf of the closed immersion $f$. 
	
	We will apply the following proof strategy. As a first step, we will prove the following proposition. 
	\begin{proposition}\label{global factorization}
		Let $g\colon\Spa{B}\to\Spa{A}$ be an smooth map of sous-perfectoid spaces. Then it factors through a closed immersion $f$
		\[\begin{tikzcd}
			\ph&\Spa{\Tate{A}{d}}\arrow[d, "\pi"]\\
			\Spa{B}\arrow[r, "g"]\arrow[ur, "f"]&\Spa{A}
		\end{tikzcd}.\]
	\end{proposition}
	
	In other words, we first show that such a factorization even exists (more generally for smooth maps).
	
	We will construct $f$ in such a way that it would factor through a closed immersion $q\colon \spa{D}\to \spa{\Tate{A}{M}}$ followed by a LOCT closed immersion $\sigma \colon\spa{B}\to\spa{D}$:
	\[ 
	\begin{tikzcd}
		& \spa{\Tate{A}{M}}\arrow[d, "\pi"] & \spa{D}\arrow[l, "q"']\\
		\spa{B}\arrow[ur, "f=q\circ\sigma"]\arrow[r, "g"]\arrow[urr, "\sigma" {xshift=70pt}] & \spa{A}. &
	\end{tikzcd}
	\]
	Then $f:=q\circ\sigma$ is a LOCT closed immersion by Lemma \ref{LOCT closed immersion satisfies COMP}. We will deduce that $f$ is a (global) closed immersion by applying \cite[Proposition IV.4.19]{FS}. That would finish the proof of Proposition \ref{global factorization}.
	
	\begin{remark}\label{local implies global remark}
		Notice that ``LOCT closed immersion implies global closed immersion" is an example of a ``local-implies-global" statement, and such statements tend to be tricky in the geometry of adic spaces. For example, whether or not the property of a Huber ring $A$ of being perfectoid (or sous-perfectoid) is local on $\mathrm{Spa}(A, A^{+})$ -- is an open problem (see \cite[Remark 3.19]{HK} and \cite[Remark 8.16]{HK} for a more precise statement). The statement that we are essentially proving is that locally finite type morphisms are globally finite type. In \cite{Hu96}, such statement is proven, and the proof is rather delicate, relying on the existence of ideal sheaves of closed immersions of adic spaces. This only works in the noetherian setting. The proof that such a statement works for \emph{us} really uses the special setting of smooth maps and the gluing result of pseudocoherent sheaves by Kedlaya-Liu from \cite{KL16}.
	\end{remark}
	
	As a second step in proving Proposition \ref{main proposition}, we will have to show that in case $g$ is \'etale, the natural map $\mathcal{I}/\mathcal{I}^{2}\to f^{*}\Omega_{\mathbb{B}^{d}_{A}/\Spa{A}}$ is an isomorphism, where $\mathcal{I}$ is the ideal sheaf of the closed immersion $f$. In fact, this readily follows from the proof of  \cite[Proposition IV.4.19]{FS}. However, it is in our plans to give an exposition of the proof of  \cite[Proposition IV.4.19]{FS}, explaining how it yields our second step on the way. To give a little more detail already right now, the isomorphism follows from the right exact sequence of the sheaves of differentials associated to a closed immersion. It is therefore important to establish the existence of this right exact sequence and explain how it works!

	The rest of this section is concerned with the proof of Proposition \ref{global factorization}, i.e., with the construction of $f$.
	
	We notice that the smooth morphism $g$ is a morphism locally of topologically finite type. Indeed, locally on the source, $g$ is a composition of an \'etale morphism followed by a projection from a ball. \'Etale morphisms, in turn, are locally finite \'etale followed by a rational localization. Each of the three is locally of topologically finite type, and so is their composition. Now, \cite[Proposition 3.8.15. ]{Hu93} asserts that this conclusion of $g$ being \emph{globally} of topologically finite type can be reached in the noetherian setting, and it relies on the existence of a good theory of closed immersions with respect to ideal sheaves. We rewrite Huber's proof of this proposition almost verbatim up until the moment where noetherian assumptions are used. This would correspond to constructing a morpihsm $f$ which is not a closed immerision, but rather a \emph{LOCT} closed immersion only. According to our strategy, this should be enough. What follows is Huber's proof\footnote{translated from German and simplified for the case of Tate rings.} of \cite[Proposition 3.8.15. ]{Hu93}, starting from the following lemma.
	
	\begin{lemma}[This is the essence of Proposition \ref{global factorization}]\label{main lemma}
		In the context of what is written right after Proposition \ref{global factorization}, one can find morphisms $f$, $q$, $\sigma$ as described in the proof strategy.
	\end{lemma}
	
	We will need a lemma.
	
	\begin{lemma}[\cite{Hu93}, Lemma 3.8.16.]\label{Huber's lemma}
		Let $\pair{A}$ be a (complete) Huber pair. There exists a neighbourhood $U\subset \Tate{A}{n}$ of $0$ such that for any $n$-tuple $(x_{1}, x_{2},\dots, x_{n})$ with $x_{i}\in X_{i}+U$ the map of Huber pairs \[ f\colon\pair{\Tate{A}{n}}\to\pair{\Tate{A}{n}} \]
		given by $f(X_{i}) = x_{i}$ is a quotient map of Huber pairs.
	\end{lemma}
	
	\begin{proof}
		Let $(A_{0}, I)$ be a pair of definition of $A$. We set $U := \Tate{I}{n}$\footnote{here $\Tate{I}{n}\subset\Tate{A}{n}$ denotes convergent power sequences with all coefficients contained in $I$, which is an open neighbourhood of $0$.}. Assume we are given an $n$-tuple $(x_{1}, x_{2},\dots, x_{n})$ with $x_{i}\in X_{i}+U$. Since $x_{i}\in \Tate{A}{n}^{+}$, we can indeed define $f$. We show that $f$ is a quotient map. Set $G_{0}:= \Tate{A_{0}}{n}$ and $G_{k} := \Tate{I^{k}}{n}$ for every positive integer $k$. Obviously, $f(G_{k})\subset G_{k}$. One can explicitly compute that for every non-negative integer $k$ and every $x\in G_{k}$ holds $x-f(x)\in G_{k+1}$. By successive approximations we conclude $f(G_{k}) = G_{k}$ (see \cite{B}, III.2.8 Cor. 2). Therefore, $f$ is an open map. Now surjectivity of $f$. We already know that $f(G_{0}) = G_{0}$, so $\Tate{A_{0}}{n}$ is in the image of $f$. Of course, $A\subset \Tate{A}{n}$ is also in the image of $f$. But the ring $\Tate{A}{n}$ is generated by $A$ and $\Tate{A_{0}}{n}$, so $f$ is surjective. Because $f$ is a map of Huber pairs, we know $f(\Tate{A}{n}^{+}) \subset \Tate{A}{n}^{+}$. Recall that $\Tate{A}{n}^{+} = (\Tate{A^{+}}{n})^{int}$, by definition. To show $f(\Tate{A}{n}^{+})^{int} = \Tate{A}{n}^{+}$, it is enough to show that $f(\Tate{A^{+}}{n})$ contains $\Tate{A^{+}}{n}$, because then $f(\Tate{A^{+}}{n}^{int}) $ contains $\Tate{A^{+}}{n}$, and hence its integral closure $f(\Tate{A^{+}}{n}^{int})^{int}$ coincides with $ \Tate{A^{+}}{n}^{int} $. We already know that $f(G_{1}) = G_{1}$, so $\Tate{I}{n}$ is cointained in $f(\Tate{A^{+}}{n})$. Of course, $ A^{+}\subset \Tate{A^{+}}{n}$ is contained in $f(\Tate{A^{+}}{n})$. Finally, each of $X_{1}, \dots, X_{n}$ is contained in $f(\Tate{A^{+}}{n})$ because $X_{i}-f(X_{i})\in G_{1}$, so there exists $y_{i}\in G_{1}$ such that $f(y_{i}) = X_{i}-f(X_{i})$, i.e., $f(y_{i}+X_{i}) = X_{i}$, where $y_{i}+X_{i} \in \Tate{A^{+}}{n}$. Since the ring $\Tate{A^{+}}{n}$ is generated by $A^{+}$, $\Tate{I}{n}$, and $X_{1},\dots X_{n}$, we are done.
	\end{proof}
	
	We start with the proof of Lemma \ref{main lemma}.
	
	\begin{proof}

	Let $g\colon\spa{B} \to \spa{A}$ correspond to the map of Huber pairs $\varphi\colon \pair{A}\to\pair{B}$. Since $g$ is locally of topologically finite type, by reduction\footnote{every rational covering of $\spa{B}$ can be refined by a \emph{standard} rational covering associated to some $f_{1},\dots, f_{n}\in B$, where $f_{1},\dots, f_{n}$ generate the unit ideal in $B$.} we know that there exist $f_{1},\dots, f_{n}\in B$ generating the unit ideal such that each Huber pair $\pair{B_{i}}$, where $B_{i}:= B\left\langle \frac{f_{1},\dots, f_{n}}{f_{i}}\right\rangle $, $i = 1, \dots, n$, is of topologically finite type over $\pair{A}$. Since we assume $B$ to be Tate, we can WLOG assume that for all $i = 1,\dots, n$ holds $f_{i}\in B^{+}$ (multiplying by an appropriate power of a pseudouniformizer). Fix a topologically nilpotent unit $\varpi\in A$. Its image $\varphi(\varpi)\in B$ is also a topologically nilpotent unit. Fix $g_{1}, \dots, g_{n} \in B^{+}$ such that $ f_{1}g_{1}+\dots+f_{n}g_{n} = \varphi(\varpi)^{\alpha}$, where $\alpha$ is some positive integer (this is possible because $f_{i}$ generate the unit ideal). Fix quotient maps of Huber pairs \[ \varphi_{i}\colon \pair{A\left\langle X_{i,1}, X_{i,2}, \dots X_{i, m}\right\rangle}\to\pair{B_{i}} \]
	(we choose $m$ large enough so that it becomes independent of $i$).
	
	Being able to precompose $\varphi_{i}$ with an arbitrary quotient map from Lemma \ref{Huber's lemma} allows us to vary the image of each $X_{i, j}$ in an open range. We can thus WLOG assume that $\varphi_{i}(X_{i, j})\in B(\frac{f_{1},\dots, f_{n}}{f_{i}})$, because the latter has dense image in $B\left\langle\frac{f_{1},\dots, f_{n}}{f_{i}}\right\rangle$. At the same time, we know that $\varphi_{i}(X_{i, j})\in B_{i}^{+}$. It is a general fact about completions that $B(\frac{f_{1},\dots, f_{n}}{f_{i}})\cap B_{i}^{+} = B(\frac{f_{1},\dots, f_{n}}{f_{i}})^{+}$, and we know that the latter is by definition $B(\frac{f_{1},\dots, f_{n}}{f_{i}})^{+} = (B^{+}[\frac{f_{1}}{f_{i}},\dots \frac{f_{n}}{f_{i}}])^{int}$. Recall that the underlying ring of $B(\frac{f_{1},\dots, f_{n}}{f_{i}})$ is $B[\frac{1}{f_{i}}]$. Since $B^{+}\subset B$ is integrally closed and $f_{i}\in B^{+}$, we have that the localization $B^{+}[\frac{1}{f_{i}}]$ is integrally closed in $ B[\frac{1}{f_{i}}]$, because taking integral closure commutes with localization. Hence, we not only have the obvious inclusion $B^{+}[\frac{f_{1}}{f_{i}},\dots \frac{f_{n}}{f_{i}}]\subset B^{+}[\frac{1}{f_{i}}]$ but also $(B^{+}[\frac{f_{1}}{f_{i}},\dots \frac{f_{n}}{f_{i}}])^{int}\subset B^{+}[\frac{1}{f_{i}}]$. It means that $\varphi_{i}(X_{i, j})\in B^{+}[\frac{1}{f_{i}}]$ and we fix $v_{i,j}\in B^{+}$ for which holds \[ \varphi_{i}(X_{i, j}) = \frac{v_{i, j}}{f_{i}^{l}} \]
	(we choose $l$ large enough so that it becomes independent of $i, j$).
	
	On the other hand, $\varphi_{i}(X_{i, j})\in (B^{+}[\frac{f_{1}}{f_{i}},\dots \frac{f_{n}}{f_{i}}])^{int}$ implies that $\varphi_{i}(X_{i, j})$ is a root of a monic polynomial \[ \varphi_{i}(X_{i, j})^{r}+a_{1}\varphi_{i}(X_{i, j})^{r-1}+\dots+a_{r-1}\varphi_{i}(X_{i, j}) + a_{r} = 0,\]
	where $a_{1},\dots, a_{r} \in B^{+}[\frac{f_{1}}{f_{i}},\dots \frac{f_{n}}{f_{i}}]$. For every $k = 1,\dots, r$ we write 
	\[ a_{k} = p_{\gamma}(\frac{f_{1}}{f_{i}},\dots, \frac{f_{n}}{f_{i}}), \]
	where $p_{\gamma} \in B^{+}[\frac{f_{1}}{f_{i}},\dots \frac{f_{n}}{f_{i}}]$ is a polynomial in $n$ variables with coefficients in $B^{+}$. We denote its coefficients by $p_{\gamma, 1},\dots, p_{\gamma, s}$  (we choose $s$ large enough so that it becomes independent of $\gamma$), so 
	\[ p_{\gamma, u}\in B^{+}\, \text{for}\, u = 1,\dots, s. \]
	
	Now, for $(i, j)\in \left\lbrace 1, \dots, n\right\rbrace\times\left\lbrace 1,\dots, m \right\rbrace  $, $d\in \left\lbrace 1,\dots, n \right\rbrace $, $e\in \left\lbrace 1,\dots, n \right\rbrace $, $\gamma = (i, j, k) \in \left\lbrace 1,\dots, n \right\rbrace\times\left\lbrace 1,\dots, m \right\rbrace\times\left\lbrace 1,\dots, r \right\rbrace$, and $u\in \left\lbrace 1,\dots, s \right\rbrace$, consider the map of Huber pairs over $\pair{A}$:
	\[  \psi\colon\pair{A\left\langle V_{i,k}, P_{\gamma, u}, F_{d}, G_{e}\right\rangle} \to \pair{B}\]
	given by $\psi(V_{i,k}) = v_{i, k}$, $\psi(P_{\gamma, u}) = p_{\gamma, u}$, $\psi(F_{d}) = f_{d}$, and $\psi(G_{e}) = g_{e}$. The map $\psi$ factors through the map $\sigma\colon \pair{D}\to \pair{B}$, where $D:=A\left\langle V_{i,k}, P_{\gamma, u}, F_{d}, G_{e}\right\rangle/\ker\psi$ is the quotient Huber ring.
	
	Denote by $\overline{v_{i, k}}$, $\overline{p_{\gamma, u}}$, $\overline{f_{d}}$, $\overline{g_{e}}$ the images of $V_{i,k}, P_{\gamma, u}, F_{d}, G_{e}$ in $D$. Since $ f_{1}g_{1}+\dots+f_{n}g_{n} - \varphi(\varpi)^{\alpha} = 0$, we know that $F_{1}G_{1}+\dots+F_{n}G_{n} - \varphi(\varpi)^{\alpha}\in\ker\psi$, and hence $\overline{f_{1}}\overline{g_{1}}+\dots+\overline{f_{n}}\overline{g_{n}} = \overline{\varpi^{\alpha}} $ in $D$, where $\overline{\varpi^{\alpha}}$ is the image of ${\varpi^{\alpha}}$ in $D$, which is a unit. Therefore, $\overline{f_{1}},\dots,\overline{f_{n}}$ generate the unit ideal in $D$. Set $D_{i}:= D\left\langle \frac{\overline{f_{1}},\dots,\overline{f_{n}}}{\overline{f_{i}}}\right\rangle$. The preimage of $\spa{D_{i}}$ is $\spa{B_{i}}$. We will show that for evey $i = 1, \dots, n$ the induced map of Huber pairs \[ \sigma_{i}\colon \pair{D_{i}}\to \pair{B_{i}} \] is a quotient map.
	
	We fix $i\in\left\lbrace 1,\dots, n \right\rbrace $. Let $E$ be the quotient Huber ring $D_{i}/\ker\sigma_{i}$ and let $\lambda\colon \pair{D_{i}}\to\pair{E}$ be the canonical quotient map of Huber pairs. We next show that for every $j\in\left\lbrace 1,\dots, m \right\rbrace $ holds \[ \lambda( \frac{\overline{v_{i,j}}}{\overline{f_{i}}^{l}}) \in E^{+}.\]
	Set $\overline{d}:= \frac{\overline{v_{i,j}}}{\overline{f_{i}}^{l}}\in D_{i}$ for brevity. For every $k\in\left\lbrace 1,\dots, r \right\rbrace $ we have defined the polynomial $p_{\gamma}$ (with $\gamma = (i,j,k)$) in $n$ variables with coefficients in $B^{+}$. Let $\overline{p_{\gamma}}$ be the polynomial obtained from $p_{\gamma}$ by substituting the coefficients $p_{\gamma, 1},\dots,p_{\gamma, s}$ with $\overline{p_{\gamma, 1}},\dots, \overline{p_{\gamma, s}}$ and the variables $\frac{f_{1}}{f_{i}},\dots, \frac{f_{n}}{f_{i}}$ with $\frac{\overline{f_{1}}}{\overline{f_{i}}},\dots, \frac{\overline{f_{n}}}{\overline{f_{i}}}$. We thus obtain an element \[ \overline{a_{k}} = \overline{p_{\gamma}}(\frac{\overline{f_{1}}}{\overline{f_{i}}},\dots, \frac{\overline{f_{n}}}{\overline{f_{i}}}) \in D_{i}.\]
	
	By construction, $\overline{d}^{r}+\overline{a_{1}}\overline{d}^{r-1}+\dots+\overline{a_{r}}$ lies in the kernel of $\sigma_{i}$, i.e.,
	\[ \lambda(\overline{d})^{r}+\lambda(\overline{a_{1}})\lambda(\overline{d})^{r-1}+\dots+\lambda(\overline{a_{r}}) = 0.\]
	From $\frac{\overline{f_{1}}}{\overline{f_{i}}},\dots, \frac{\overline{f_{n}}}{\overline{f_{i}}}\in D_{i}^{+}$ and $\overline{p_{\gamma, u}}\in D_{i}^{+}$ we get that $\overline{a_{k}} \in D_{i}^{+}$. Since $E^{+}$ is integrally closed, we achieve $\lambda(\overline{d})\in E^{+}$, as we wished.
	
	We can thus define the following map of Huber pairs \[ \mu\colon\pair{A\left\langle X_{i,1},\dots, X_{i, m}\right\rangle} \to \pair{E}\]
	given by $\mu(X_{i, j}) = \lambda(\frac{\overline{v_{i,j}}}{\overline{f_{i}}^{l}})$ for $j = 1, \dots, m$. The map of Huber pairs \[ \sigma_{i}\colon \pair{D_{i}}\to \pair{B_{i}} \] factors through the map of Huber pairs $\tau\colon \pair{E}\to \pair{B_{i}}$. We consider the following diagram:
	\[ 
	\begin{tikzcd}
		\pair{D_{i}}\arrow[r, "\sigma_{i}"]\arrow[d, "\lambda"] & \pair{B_{i}}\\
		\pair{E}\arrow[ur, "\tau"]& \pair{A\left\langle X_{i,1},\dots, X_{i, m}\right\rangle}\arrow[l, "\mu"]\arrow[u, "\varphi_{i}"]
	\end{tikzcd}
	 \]
	By construction, the diagram is commutative. By assumption, $\varphi_{i}$ is a quotient map of Huber pairs. It follows that $\tau$ is a quotient map of Huber pairs. Consequently, $\sigma_{i}$ is a quotient map of Huber pairs.
	
	In the context of the proof strategy outlined in the beginning of this section, we let $M$ to be the total number of variables $V_{i,k}, P_{\gamma, u}, F_{d}, G_{e}$. We let $f\colon\spa{B}\to\spa{\Tate{A}{M}}$ be the morphism corresponding to the map of Huber pairs $ \psi\colon\pair{A\left\langle V_{i,k}, P_{\gamma, u}, F_{d}, G_{e}\right\rangle} \to \pair{B}$. We set $D:=A\left\langle V_{i,k}, P_{\gamma, u}, F_{d}, G_{e}\right\rangle/\ker\psi$. We let $\sigma \colon\spa{B}\to\spa{D}$ be the morphism corresponding to the map of Huber pairs $\sigma\colon\pair{D}\to\pair{B}$. We let $q\colon \spa{D}\to \spa{\Tate{A}{M}}$ to be the closed immersion corresponding to the canonical quotient map of Huber pairs \[ \pair{\Tate{A}{M}}\to\pair{D}. \] The proof of Lemma \ref{main lemma} is finished.
	
	\end{proof}
	
	\subsection{Global complete intersection}\label{subsection global complete intersection}
	
	In this subsection we show that we can refine our choice of the closed immersion from Proposition \ref{main proposition} such that the ideal $\mathcal{I}$ is generated by exactly $n$ elements, where $n$ is the number of variables of $\spa{\Tate{A}{n}}$, i.e, we prove Proposition \ref{upgraded main proposition}. In fact, its proof is analogous to the proof of Proposition \ref{proposition 1.7.1} (see \cite[Proposition 1.7.1]{Hu96}), which in turn is analogous to the proof of Proposition \ref{proposition 0} (see \cite[\href{https://stacks.math.columbia.edu/tag/00U9}{Tag 00U9}]{Stacks}). However, in our case we face a technical difficulty concerning the availability of a theory of ``quasicoherent" sheaves on adic spaces. To deal with this technical difficulty, in what follows, we will use the theory of pseudocoherent sheaves on adic spaces, as developed in \cite{KL16}. In the next subsubsection, we collect the statements from this theory that we will use, for reference. We collect some important statements from \cite{FS}, too.
		
	\subsubsection{Pseudocoherent sheaves}
		
	\begin{lemma}[{\cite[Lemma IV.4.13]{FS}}]\label{regular sequence}
		Let $X = \Spa{\Tate{R}{n}}$ where $R$ is a sous-perfectoid Tate ring, let $Y = \Spa{S}$ where $S$ is a sous-perfectoid Tate ring, and let $f\colon Y\to X$ be a smooth map. Then $X_{1},\dots, X_{n}$ define a regular sequence on $S$ and $(X_{1},\dots, X_{n})S\subset S$ is a closed ideal.
	\end{lemma}
	\begin{proposition}[{\cite[Proposition IV.4.15]{FS}}]\label{"conormal" sequence}
		Let $f_{i}\colon Y_{i}\to X$, $i = 1,2$, be smooth maps of sous-perfectoid adic spaces, and let $g\colon Y_{1}\to Y_{2}$ be a map over $X$.
		\begin{itemize}
			\item[i)] If $g$ is smooth, then the sequence \[ 0\to g^{*}\Omega^{1}_{Y_{2}/X}\to\Omega^{1}_{Y_{1}/X}\to\Omega^{1}_{Y_{1}/Y_{2}}\to 0 \] is exact.
			\item[ii)] Conversely, if $ g^{*}\Omega^{1}_{Y_{2}/X}\to\Omega^{1}_{Y_{1}/X}$ is a locally split injection, then $g$ is smooth.
			\item[\ph] In particular, if $g^{*}\Omega^{1}_{Y_{2}/X}\to\Omega^{1}_{Y_{1}/X}$ is an isomorphism, then $g$ is \'etale.
		\end{itemize}
	\end{proposition}
	\begin{proposition}[{\cite[Proposition IV.4.19]{FS}}]\label{closed immersion behaves well}
		Let $f\colon Y\to X$, $f\colon Y^{\prime}\to X$ be smooth maps of sous-perfectoid adic spaces, and let $g\colon Y\to Y^{\prime}$ be a map over $X$. The following are equivalent.
		\begin{itemize}
			\item[i)] There is a cover of $Y^{\prime}$ by open affinoid $V^{\prime} = \Spa{S^{\prime}}$ such that $V = Y\times_{Y^{\prime}}V^{\prime} = \Spa{S}$ is affinoid and $\pair{S^{\prime}}\to \pair{S}$ is a quotient map of Huber pairs.
			\item[i))] For any open affinoid $V^{\prime} = \Spa{S^{\prime}}\subset Y^{\prime}$, the preimage $V = Y\times_{Y^{\prime}}V^{\prime} = \Spa{S}$ is affinoid and $\pair{S^{\prime}}\to \pair{S}$ is a quotient map of Huber pairs.
		\end{itemize}
		Moreover, in this case the ideal sheaf $\mathcal{I_{Y\subset Y^{\prime}}}\subset \mathcal{O}_{Y^{\prime}}$ is pseudocoherent, and locally generated by sections $f_{1},\dots, f_{r}\in\mathcal{O}_{Y^{\prime}}$ such that $df_{1},\dots, df_{r}\in\Omega^{1}_{Y^{\prime}/X}$ can locally be extended to a basis.
	\end{proposition}
	
	Let $A$ be a ring. We denote the category of $A$-modules by $ \mathrm{Mod}_{A}$. We denote the category of finite projective $A$-modules by $\mathrm{FPMod}_{A}$.
	\begin{definition}
		An $A$-module $M$ is called \textit{pseudocoherent} if it admits a resolution (possibly infinite) by finite projective $A$-modules \[ \to P_{1}\to P_{0} \to A\to 0. \] 
	\end{definition}
	
	\begin{theorem}[{\cite[Theorem 1.4.2]{K17}}]\label{equivalence of categories finite projective}
		Let $\pair{A}$ be a Huber pair and let $X = \Spa{A}$. If $A$ is sheafy, the functors $\mathrm{FPMod_{A}}\to\mathrm{Vec_{X}}$ given by $M\mapsto \widetilde{M}$ is an exact equivalence of categories, with quasi-inverse taking $\mathcal{F}$ to $\mathcal{F}(X)$. 
	\end{theorem}
	
	\begin{definition}
		Let $A$ be a topological ring. Define the cateogry $\mathrm{PCoh}_{A}$ to be the full subcategory of $\mathrm{Mod}_{A}$ spanned by pseudocoherent $A$-modules that are complete for their natural topology.
	\end{definition}
	
	 It is true that any arrow in $\mathrm{PCoh}_{A}$ is automatically continuous.
	
	\begin{theorem}[{\cite[Theorem 1.4.13]{K17}}]\label{pseudoflatness of rational localizations}
		Let $\pair{A}$ be a Huber pair and let $\pair{A}\to\pair{B}$ be a rational localization. If $A$ is sheafy, base extension from $A$ to $B$ defines an exact functor $\mathrm{PCoh}_{A}\to \mathrm{PCoh}_{B}$.
	\end{theorem}
	
	\begin{theorem}[{\cite[Theorem 1.4.17]{K17}}]\label{equivalence of categories pseudocoherent}
		Let $\pair{A}$ be a Huber pair and let $X = \Spa{A}$. If $A$ is sheafy, the functors $\mathrm{PCoh_{A}}\to\mathrm{PCoh_{X}}$ given by $M\mapsto \widetilde{M}$ is an exact equivalence of categories, with quasi-inverse taking $\mathcal{F}$ to $\mathcal{F}(X)$. 
	\end{theorem}
	
	\subsubsection{Proof of global complete intersection}
	
	We recall the setup: 
	
	\[\begin{tikzcd}
		\ph&\spa{\Tate{A}{n}}=\mathbb{B}^{n}_{A}\arrow[d, "\pi"]\\
		\spa{B}\arrow[r, "g"]\arrow[ur, "f"]&\spa{A},
	\end{tikzcd}\]
	where $g\colon\spa{B}\to\spa{A}$ is an \'etale map of affinoid sous-perfectoid spaces, $f$ is a global closed immersion with ideal sheaf $\mathcal{I}$, and the natural map $\mathcal{I}/\mathcal{I}^{2}\to f^{*}\Omega^{1}_{\mathbb{B}^{n}_{A}/\spa{A}}$ is an isomorphism, by Proposition \ref{"conormal" sequence}. By Proposition \ref{closed immersion behaves well}, the ideal sheaf $\mathcal{I}$ is pseudocoherent, and is therefore determined by its global sections $I:=\mathcal{I}(\spa{\Tate{A}{n}})$, i.e., $\mathcal{I}=\widetilde{I}$. Of course, we know that $B = \Tate{A}{n}/I$.
	
	Since $I$ and $ \Tate{A}{n}$ are pseudocoherent $ \Tate{A}{n}$-modules, by the two-out-of-three property we get that $ \Tate{A}{n}/I$ is pseudocoherent, too. Since $I\subset \Tate{A}{n}$ is a closed ideal, both $I$ and $ \Tate{A}{n}/I$ actually belong to $\mathrm{PCoh_{\Tate{A}{n}}}$. 
	
	Let \[ \Tate{A}{n}^{\oplus m_{1}}\to\Tate{A}{n}^{\oplus m_{2}}\to I\to 0 \] be a right exact sequence of $\Tate{A}{n}$-modules. Note that all of them belong to $\mathrm{PCoh_{\Tate{A}{n}}}$. By the exact equivalence of categories from Theorem \ref{equivalence of categories pseudocoherent}, we get a right exact sequence of $\mathcal{O}_{\Tate{A}{n}}$-modules \[ \mathcal{O}_{\Tate{A}{n}}^{\oplus m_{1}}\to \mathcal{O}_{\Tate{A}{n}}^{\oplus m_{2}}\to \mathcal{I} \to 0. \]  Since the pullback functor $f^{*}\colon \mathrm{Mod}_{\mathcal{O}_{\Tate{A}{n}}}\to\mathrm{Mod}_{\mathcal{O}_{B}}$ is right exact, we get a right exact sequence of $\mathcal{O}_{B}$-modules \[ \mathcal{O}_{B}^{\oplus m_{1}}\to \mathcal{O}_{B}^{\oplus m_{2}}\to f^{*}\mathcal{I} \to 0. \] From what we already know, $ f^{*}\mathcal{I} $ is free of rank $n$, in particular, it is a vector bundle. By the exact equivalence of categories from Theorem \ref{equivalence of categories finite projective}, taking global sections yields a right exact sequence of finite projective $B$-modules \[ {B}^{\oplus m_{1}}\to {B}^{\oplus m_{2}}\to f^{*}\mathcal{I}(\spa{B}) \to 0. \] Clearly, the first arrow from above is the tensor of the first arrow from \[ \Tate{A}{n}^{\oplus m_{1}}\to\Tate{A}{n}^{\oplus m_{2}}\to I\to 0 \] with $B$, and since tensor product preserves cokernels, the global sections $f^{*}\mathcal{I}(\spa{B})$ is the tensor of $I$ with $B$, i.e., $f^{*}\mathcal{I}(\spa{B}) = I/I^{2}$, in particular, $I/I^{2}$ is a finite projective $B$-module. By the same exact equivalence of categories from Theorem \ref{equivalence of categories finite projective}, $f^{*}\mathcal{I} = \widetilde{I/I^{2}}$.

	We get an isomorphism $I/I^{2}\cong B^{\oplus n}$. Choose elements $f_{1},\dots, f_{n}\in I$ whose residues $\overline{f_{1}},\dots, \overline{f_{1}}\in I/I^{2}$ form a basis as of an $B$-module. By multiplying by an appropriate power of a pseudouniformizer, we can choose $f_{1},\dots, f_{n}\in I$ such that they lie inside $\Tate{A}{n}^{+}$, and we do that. Note that we cannot guarantee neither that $f_{1},\dots, f_{n}$ generate $I$ as a $\Tate{A}{n}$-module nor that $df_{1},\dots, df_{n} \in \Omega^{1}_{\mathbb{B}^{n}_{A}/\spa{A}}\cong \mathcal{O}_{\Tate{A}{n}}^{\oplus n}$ generate $\Omega^{1}_{\mathbb{B}^{n}_{A}/\spa{A}}$ as a $\Tate{A}{n}$-module\footnote{the latter is equivalent to $df_{1},\dots, df_{n}$ forming a basis of $\Omega^{1}_{\mathbb{B}^{n}_{A}/\spa{A}}$.}. However, by our choice, we know that the residues $\overline{f_{1}},\dots, \overline{f_{n}}$ generate $I/I^{2}$, while the residues $\overline{df_{1}},\dots, \overline{df_{n}}\in \Omega^{1}_{\mathbb{B}^{n}_{A}/\spa{A}}/I\Omega^{1}_{\mathbb{B}^{n}_{A}/\spa{A}}\cong f^{*}\Omega^{1}_{\mathbb{B}^{n}_{A}/\spa{A}}\cong \mathcal{O}_{B}^{\oplus n}$ generate $\Omega^{1}_{\mathbb{B}^{n}_{A}/\spa{A}}/I\Omega^{1}_{\mathbb{B}^{n}_{A}/\spa{A}}.$ Hence, for a finite $\Tate{A}{n}$-module $M$ equal to either $I$ or $\Omega^{1}_{\mathbb{B}^{n}_{A}/\spa{A}}$, we can apply the following version of Nakayama's lemma, see \cite[\href{https://stacks.math.columbia.edu/tag/00DV}{Tag 00DV}]{Stacks}, bullet (7).
	
	\begin{lemma}\label{Nakayama's lemma}
		Let $R$ be a ring and $I\subset R$ be an ideal. Let $M$ be a finite $R$-module. If elements $x_{1},\dots, x_{n}\in M$ generate $M/IM$, then there exists $g\in 1+I$ such that $x_{1},\dots, x_{n}$ generate $M[\frac{1}{g}]$ over $R[\frac{1}{g}]$.
	\end{lemma}
	
	Let $g^{\prime}\in 1+I$ be the element resulting from applying the above lemma to $M = I$, and let $g^{\prime\prime}\in 1+I$ be the element resulting from applying the above lemma to $M = \Omega^{1}_{\mathbb{B}^{n}_{A}/\spa{A}}$. Set $g = g^{\prime}g^{\prime\prime}$. Clearly $g\in 1+I$, and since localizing at $g$ is the same as localizing consecutively first at $g^{\prime}$ and then at $g^{\prime\prime}$\footnote{alternatively, we just know that $D(g)=D(g^{\prime})\cap D(g^{\prime\prime})$ on $\mathrm{Spec}\,\Tate{A}{n}$.}, we conclude that (simultaneously) $f_{1},\dots, f_{n}$ generate $I[\frac{1}{g}]$ and $df_{1},\dots, df_{n}$ generate ${\Omega^{1}_{\mathbb{B}^{n}_{A}/\spa{A}}}[\frac{1}{g}]$ over $\Tate{A}{n}[\frac{1}{g}]$.
	
	The elements $f_{1},\dots, f_{n}\in \Tate{A}{n}^{+}$ induce a map over $\spa{A}$ \[ \spa{\Tate{A}{n}} \to \mathrm{Spa}\,A\left\langle F_{1},\dots, F_{n}\right\rangle \] sending $F_{i}$ to $f_{i}$. Composing with rational localization at $g$, we get maps over $\spa{A}$ \[ \spa{\Tate{A}{n}}\left\langle \frac{1}{g}\right\rangle \to\spa{\Tate{A}{n}} \to \mathrm{Spa}\,A\left\langle F_{1},\dots, F_{n}\right\rangle. \] 
	
	Applying Lemma \ref{"conormal" sequence} to $X = \spa{A}$, $Y_{1} = \spa{\Tate{A}{n}}\left\langle \frac{1}{g}\right\rangle$, $ Y_{2} = \spa{\Tate{A}{n}}$, and knowing that rational localizations are \`etale, we see that \[ \Omega^{1}_{Y_{1}/X} \cong \Omega^{1}_{Y_{2}/X}\otimes_{\Tate{A}{n}}\Tate{A}{n}\left\langle \frac{1}{g}\right\rangle.\] Since $df_{1},\dots, df_{n}$ generate ${\Omega^{1}_{Y_{2}/X}}[\frac{1}{g}]$ over $\Tate{A}{n}[\frac{1}{g}]$, they will also generate \[\Omega^{1}_{Y_{1} /X} = \Omega^{1}_{Y_{2}/X}\otimes_{\Tate{A}{n}}\Tate{A}{n}\left\langle \frac{1}{g}\right\rangle = \Omega^{1}_{Y_{2}/X}[\frac{1}{g}]\otimes_{\Tate{A}{n}[\frac{1}{g}]}\Tate{A}{n}\left\langle \frac{1}{g}\right\rangle\] over $\Tate{A}{n}\left\langle \frac{1}{g}\right\rangle$. In particular, applying (the converse of) Lemma \ref{"conormal" sequence} to the composition over $\spa{A}$ \[ \spa{\Tate{A}{n}}\left\langle \frac{1}{g}\right\rangle \to \mathrm{Spa}\,A\left\langle F_{1},\dots, F_{n}\right\rangle, \] we get that this composition is \'etale.
	
	By Lemma \ref{regular sequence}, the elements $f_{1}, \dots, f_{n}$ define a regular sequence on $\Tate{A}{n}\left\langle \frac{1}{g}\right\rangle $ and the ideal $(f_{1},\dots, f_{n}) \subset \Tate{A}{n}\left\langle \frac{1}{g}\right\rangle$ generated by them is already closed. In fact, we do not even need to use Lemma \ref{regular sequence} to infer that this ideal is closed, as we can do the following instead. To avoid confusion, re-denote this ideal by $ (f_{1},\dots, f_{n})\left\langle \frac{1}{g}\right\rangle $. By our initial choice, $ (f_{1},\dots, f_{n})[\frac{1}{g}] = I[\frac{1}{g}]$ inside $\Tate{A}{n}[\frac{1}{g}]$. Therefore, $ (f_{1},\dots, f_{n})\left\langle \frac{1}{g}\right\rangle  = I\left\langle \frac{1}{g}\right\rangle $ inside $\Tate{A}{n}\left\langle \frac{1}{g}\right\rangle $, where $I\left\langle \frac{1}{g}\right\rangle $ is the ideal that all elements of $I$\footnote{equivalently, of $ I[\frac{1}{g}]$.} generate inside $\Tate{A}{n}\left\langle \frac{1}{g}\right\rangle $. As we have seen before, we have a short exact sequence in $\mathrm{PCoh_{\Tate{A}{n}}}$: \[ 0\to I \to \Tate{A}{n}\to \Tate{A}{n}/I\to 0. \] By Theorem \ref{pseudoflatness of rational localizations} applied to the rational localization $\Tate{A}{n}\to \Tate{A}{n}\left\langle \frac{1}{g}\right\rangle $, we obtain an exact sequence in $\mathrm{PCoh_{\Tate{A}{n}\left\langle \frac{1}{g}\right\rangle }}$: \[ \hspace*{-53pt}0\to I\otimes_{\Tate{A}{n}}\Tate{A}{n}\left\langle \frac{1}{g}\right\rangle \to \Tate{A}{n}\left\langle \frac{1}{g}\right\rangle \to (\Tate{A}{n}/I)\otimes_{\Tate{A}{n}}\Tate{A}{n}\left\langle \frac{1}{g}\right\rangle\to 0.  \] It is true for rings in general that $I\left\langle \frac{1}{g}\right\rangle $ (as defined beforehand) is the image of the first arrow of the latter sequence. This justifies writing simply \[ I\left\langle \frac{1}{g}\right\rangle = I\otimes_{\Tate{A}{n}}\Tate{A}{n}\left\langle \frac{1}{g}\right\rangle.\] Since all modules in the latter sequence are complete (by definition of $\mathrm{PCoh}$), the ideal $I\left\langle \frac{1}{g}\right\rangle$ is closed, but we know that this is the same as the ideal $(f_{1},\dots, f_{n})\left\langle \frac{1}{g}\right\rangle$. Finally, the short exact sequence also yields the equality \[\hspace*{-23pt} \Tate{A}{n}/I\left\langle \frac{1}{g}\right\rangle:= (\Tate{A}{n}/I)\otimes_{\Tate{A}{n}}\Tate{A}{n}\left\langle \frac{1}{g}\right\rangle = \Tate{A}{n}\left\langle \frac{1}{g}\right\rangle/I\left\langle \frac{1}{g}\right\rangle.\]
	
	Actually, when the element $g$ is treated like an element of $\Tate{A}{n}/I$ (by abuse of notation), the expression $\Tate{A}{n}/I\left\langle \frac{1}{g}\right\rangle$ starts making sense as the rational localization of the Huber ring $\Tate{A}{n}/I$ at $g$. It is clear that this interpretation coincides with what we already have. But $g\in 1+I$, and so it equals $1$ in $\Tate{A}{n}/I$, yielding $\Tate{A}{n}/I\left\langle \frac{1}{g}\right\rangle = \Tate{A}{n}/I.$ 
	 
	 It allows us to write
	 
	 \begin{align*}
		B = \Tate{A}{n}/I = \Tate{A}{n}/I\left\langle \dfrac{1}{g}\right\rangle &= \Tate{A}{n}\left\langle \dfrac{1}{g}\right\rangle/I\left\langle \dfrac{1}{g}\right\rangle\\ &= \Tate{A}{n}\left\langle \dfrac{1}{g}\right\rangle/(f_{1},\dots, f_{n})\left\langle \dfrac{1}{g}\right\rangle\\ &=  \Tate{A}{n}\left\langle \dfrac{1}{g}\right\rangle/(f_{1},\dots, f_{n})\\ &= {A\left\langle X_{1},\dots, X_{n}, X\right\rangle/(gX-1)/(f_{1},\dots, f_{n})}\\ &= {A\left\langle X_{1},\dots, X_{n}, X\right\rangle/(f_{1},\dots, f_{n}, gX-1)},
	 \end{align*}
	where in the second to last equality we are using the fact that $A$ is uniform to say that the ideal generated by $gX-1$ is already closed. Thus, we have shown that we can choose a presentation of $B$ where the ideal is generated by the correct number of elements. If we are allowed to use that $f_{1},\dots, f_{n}$ is a regular sequence, then $f_{1},\dots, f_{n}, gX-1$ is a regular sequence, and the ideal is moreover generated by a regular sequence.
	
	Here is a commutative diagram of the situation, in which every square is cartesian (everything is over $\spa{A}$):
	\[ \begin{tikzcd}
		V(I) \arrow[r, "id"] \arrow[d, "\acute{e}tale"]& V(I) \arrow[d]\arrow[dr, bend left, "\acute{e}tale"]& \ph\\
		\spa{\Tate{A}{n}\left\langle \dfrac{1}{g}\right\rangle}/(f_{1},\dots, f_{n}) \arrow[r, "\acute{e}tale"] \arrow[d]& V(f_{1},\dots, f_{n})\arrow[r]\arrow[d] & \spa{A}\arrow[d, "0"]\\
		\spa{\Tate{A}{n}\left\langle \dfrac{1}{g}\right\rangle}\arrow[r, "\acute{e}tale"] \arrow[rr, bend right, "\acute{e}tale"]& \spa{\Tate{A}{n}}\arrow[r] & \spa{A\left\langle F_{1},\dots, F_{n}\right\rangle}\\
	\end{tikzcd}. \]
	For example, $V(I) = \spa{B}$, and it is \'etale over $\spa{A}$ by assumption. The arrow $V(I) \to \spa{\Tate{A}{n}\left\langle \frac{1}{g}\right\rangle}/(f_{1},\dots, f_{n}) $ is \'etale by the cancellation property of \'etale maps, as both $V(I)$ and $\spa{\Tate{A}{n}\left\langle \frac{1}{g}\right\rangle}/(f_{1},\dots, f_{n}) $ are \'etale over $\spa{A}$. In fact, as we have shown, this arrow is an isomorphism.
	
	\section{Proof of noetherian approximation}\label{section proof of noetherian approximation}
	
	In this section we prove Proposition \ref{noetherian approximation}.  For this, we need to state a fuller version\footnote{for the fullest (and original) version, see \cite[Proposition 1.7.1]{Hu96}, and see \cite[Assumption 1.1.1]{Hu96} for the noetherian assumpitons that are in effect.} of Proposition 1.7.1 from \cite{Hu96}:
	
	\begin{proposition}\label{fuller proposition 1.7.1}
		Let $\pair{A}, \pair{B}$ be Huber pairs with $A, B$ complete Tate and strongly noetherian. Let $g\colon\Spa{B}\to\Spa{A}$ be a map of affinoid adic spaces. Then $g$ is \'etale if and only if there exists a presentation \[ \pair{B}\cong \pair{\Tate{A}{n}/(f_{1},\dots,f_{n})} \] such that the determinant of the associated Jacobian matrix $\mathrm{det}(\frac{\partial f_{i}}{\partial X_{j}})_{1\leqslant i, j\leqslant n}$ is a unit in $B$.
		
		In this case, one can choose $f_{1},\dots, f_{n}$ to be polynomials, i.e., elements of $A[X_{1},\dots, X_{n}]$.
	\end{proposition}
	
	We claim that the analogous refinement holds in our setting (which we prove slightly later):
	
	\begin{proposition}\label{refined upgraded main proposition}
		Let $g\colon\Spa{B}\to\Spa{A}$ be a map of affinoid sous-perfectoid spaces. Then $g$ is \'etale if and only if there exists a presentation \[ \pair{B}\cong \pair{\Tate{A}{n}/(f_{1},\dots,f_{n})} \] such that the determinant of the associated Jacobian matrix $\mathrm{det}(\frac{\partial f_{i}}{\partial X_{j}})_{1\leqslant i, j\leqslant n}$ is a unit in $B$.
		
		In this case, one can choose $f_{1},\dots, f_{n}$ to be polynomials, i.e., elements of $A[X_{1},\dots, X_{n}]$.
	\end{proposition}
	
	In addition, we invoke Corollary 1.7.2 from \cite{Hu96}:
	
	\begin{corollary}[{\cite[Corollary 1.7.2]{Hu96}}]
		Let $\pair{A}$ be a Huber pair with $A$ complete Tate and strongly noetherian. Let \[ \pair{B}:= \pair{\Tate{A}{n}/(f_{1},\dots,f_{n})} \] such that the determinant of the associated Jacobian matrix $\mathrm{det}(\frac{\partial f_{i}}{\partial X_{j}})_{1\leqslant i, j\leqslant n}$ is a unit in $B$. Then there exists a neighbourhood $S$ of $0$ in $\Tate{A}{n}$ such that if $f_{1}^{\prime} \in f_{1} + S, \dots, f_{n}^{\prime}\in f_{n}+S$ then the determinant of the Jacobian matrix $\mathrm{det}(\frac{\partial f_{i}^{\prime}}{\partial X_{j}})_{1\leqslant i, j\leqslant n}$ is a unit in \[ \pair{B^{\prime}}:= \pair{\Tate{A}{n}/(f_{1}^{\prime},\dots,f_{n}^{\prime})} \] and there exists an isomorphism $ \Spa{B}\cong \Spa{B^{\prime}}$ over $\Spa{A}$.
	\end{corollary}
	
	We claim that the analogous corollary holds in our setting (which we prove slightly later):
	
	\begin{corollary}[{\cite[Corollary 1.7.2]{Hu96}}]\label{corollary}
		Let $g\colon\Spa{B}\to\Spa{A}$ be a map of affinoid sous-perfectoid spaces, where $A$ is Tate. Choose a presentation \[ \pair{B}:= \pair{\Tate{A}{n}/(f_{1},\dots,f_{n})} \] such that the determinant of the associated Jacobian matrix $\mathrm{det}(\frac{\partial f_{i}}{\partial X_{j}})_{1\leqslant i, j\leqslant n}$ is a unit in $B$. Then there exists a neighbourhood $S$ of $0$ in $\Tate{A}{n}$ such that if $f_{1}^{\prime} \in f_{1} + S, \dots, f_{n}^{\prime}\in f_{n}+S$ then the determinant of the Jacobian matrix $\mathrm{det}(\frac{\partial f_{i}^{\prime}}{\partial X_{j}})_{1\leqslant i, j\leqslant n}$ is a unit in \[ \pair{B^{\prime}}:= \pair{\Tate{A}{n}/(f_{1}^{\prime},\dots,f_{n}^{\prime})} \] and there exists an isomorphism $ \Spa{B}\cong \Spa{B^{\prime}}$ over $\Spa{A}$.
	\end{corollary}
	
	We now prove Proposition \ref{refined upgraded main proposition} and Corollary \ref{corollary}.
	\begin{proof}
		We claim that almost nothing has to be changed from Huber's proof of Proposition \ref{fuller proposition 1.7.1} in order to get Proposition \ref{refined upgraded main proposition}. In fact, the only place where noetherian assumptions are used in the proof of Proposition 1.7.1 from \cite{Hu96} is inside part (3) of the proof of the implication ((ii) $\implies$ (iii)). There, one needs to show that some endomorphism $f\colon B\to B$ of a complete Huber ring $B$ is bijective; surjectivity of $f$ is shown by arguing that the endomorphism $f_{0}\colon B_{0}\to B_{0}$ induced on a ring of definition $B_{0}$ is surjective, by applying \cite[III.2.9 Cor. 3]{B}, while injectivity of $f$ becomes automatic as every surjective endomorphism of a noetherian ring is injective. However, \cite[III.2.9 Cor. 3]{B} already implies bijectivity of $f_{0}\colon B_{0}\to B_{0}$ (not merely surjectivity). Consequently, we can deduce bijectivity of $f$ (without using the noetherian assumption) by multiplying/dividing by a power of a pseudouniformizer of $B$ in case it is Tate, which it is by our assumption that $A$ is Tate. The rest of Huber's proof does not use any noetherian assumptions, and we have nothing to add or subtract.
	\end{proof}
	
	We now prove Proposition \ref{noetherian approximation}.
	
	\begin{proof}
		Assume we are in the context of Proposition \ref{noetherian approximation}, with the same notation. Choose a presentation  \[ \pair{B}\cong \pair{\Tate{R}{n}/(f_{1},\dots,f_{n})} \] with $f_{1},\dots, f_{n}\in R[X_{1},\dots, X_{n}]$, which is possible by Proposition \ref{refined upgraded main proposition}. By multiplying with $\varpi$, we WLOG assume that the coefficients of $f_{1},\dots, f_{n}$ belong to $R^{+}$. By Corollary \ref{corollary}, we are allowed to perturb $f_{1}, \dots, f_{n}$ without changing the quotient $\pair{\Tate{R}{n}/(f_{1},\dots,f_{n})}$. Since $\varinjlim_{i} R_{i}^{+} $ is dense in $R^{+}$, by perturbing we may assume that all coefficients of $f_{1},\dots, f_{n}$ are elements of $R_{i^{\prime}}^{+}$ for finitely many $i^{\prime}\in I$. Then all the coefficients of  $f_{1},\dots, f_{n}$ are already defined over some single $R_{i}^{+}$ for some $i\in I$, by filteredness of $I$. Setting \[ \pair{B_{i}}:= \pair{\Tate{R_{i}}{n}/(f_{1},\dots,f_{n})} \] finishes the proof.
	\end{proof}
	
	\section{The sheaf of differentials}\label{section the sheaf of differentials}
	
	In this section we will give a local description of the sheaf of differentials defined in \cite{FS} in terms of the module of differentials. We prove Proposition \ref{description of the sheaf of differentials}; we remind its statement below. 
	
	\begin{proposition*}[Proposition \ref{description of the sheaf of differentials}]
		Let $f\colon Y\to X$ be a smooth map of sous-perfectoid adic spaces. For any open affinoids $\spa{B}\subset Y$ and $\spa{A}\subset X$ with $f(\spa{B})\subset \spa{A}$ holds \[ \Omega^{1}_{Y/X}\mid_{\spa{B}}\cong \Omega^{1}_{\spa{B}/\spa{A}}\cong \widehat{\Omega^{1}_{B/A}}, \]
		where $\widehat{\Omega^{1}_{B/A}}$ is the usual module of differentials completed with respect to a suitable topology (defined later).
	\end{proposition*}
	
	In what follows we perform a relatively long computation that proves the second equality of Proposition \ref{description of the sheaf of differentials}. This computation consists of two main steps. We will be able to state the first step after a couple of paragraphs.
	
	By definition, $\Omega^{1}_{\spa{B}/\spa{A}}$ is the pullback of the ideal sheaf of the global closed immersion \[ \spa{B}\to \spa{\widehat{B\otimes_{A} B}}. \] Since both $\spa{B}$ and $ \spa{\widehat{B\otimes_{A} B}}$ are smooth over $\spa{A}$, we can apply Lemma \ref{closed immersion behaves well} to deduce that the ideal sheaf $\mathcal{J}$ is pseudocoherent, and if we denote its global sections by $J:= \mathcal{J}(\spa{\widehat{B\otimes_{A}B}})$, its pullback is equal to $\widetilde{J/J^{2}}$. In other words, thanks to the theory of pseudocoherent sheaves, we can reduce to topological algebra (no more sheaves). We record the short exact sequence 
	
	\begin{equation}\label{equation}
		0\to J\to \widehat{B\otimes_{A}B} \to B\to 0.
	\end{equation}

	In order to state the first step, we need to fix some notation. Choose rings of definition $A_{0}\subset A$ and $B_{0}\subset B$ such that $A_{0}$ maps inside $B_{0}$. Set $C_{0}:= B_{0}\otimes_{A_{0}}B_{0}$ and $C := B\otimes_{A}B$ for brevity. Let $I_{0}$ be the kernel of the surjection $C_{0} \to B_{0}$, so that we have a short exact sequence of topological $C_{0}$-modules \[ 0 \to I_{0}\to C_{0}\to B_{0}\to 0. \] Let $\varpi\in A_{0}$ be a pseudouniformizer in the Tate ring $A$. Since pseudouniformizers map to pseudouniformizers, $\varpi$ can be also viewed as an element of $B_{0}$ and $C_{0}$ (by abuse of notation). Thus, $A = A_{0}[\frac{1}{\varpi}]$, $B = B_{0}[\frac{1}{\varpi}]$, and $C = C_{0}[\frac{1}{\varpi}]$. Denote $I := I_{0}[\frac{1}{\varpi}]$ as a topological $C_{0}$-module. We will define and explain the topology on $I$ later.
	
	The first step of this lengthy computation is to prove the equality $J=\widehat{I}$. This will be achieved in three pages.
	
	Consider the following commutative diagram of abelian groups
	
	\[ \begin{tikzcd}
		\ph&0&0&0&\ph\\
		\ph& I_{0}/I_{0}[\varpi^{\infty}]\arrow[r]\arrow[u]&  C_{0}/C_{0}[\varpi^{\infty}] \arrow[r]\arrow[u]&  B_{0}\arrow[u]& \ph\\
		0\arrow[r]&I_{0}\arrow[r]\arrow[u]&C_{0}\arrow[r]\arrow[u]&B_{0}\arrow[r]\arrow[u, "\text{id}"]&0\\
		0\arrow[r]&I_{0}[\varpi^{\infty}]\arrow[r, "\cong"]\arrow[u]&C_{0}[\varpi^{\infty}]\arrow[r]\arrow[u]&0\arrow[r]\arrow[u]&0\\
		\ph&0\arrow[u]&0\arrow[u]&0,\arrow[u]&\ph
	\end{tikzcd} \] where $(-)[\varpi^{\infty}]$ denotes $\varpi$-power torsion. In this diagram, vertical columns are exact by the obvious reason, the middle row is exact by assumption, and the bottom row is exact because $B_{0}\subset B$ is $\varpi$-torsion free. By diagram chase, we get the exactness of the sequence
	\[ \begin{tikzcd}
		0\arrow[r]& I_{0}/I_{0}[\varpi^{\infty}]\arrow[r]&  C_{0}/C_{0}[\varpi^{\infty}] \arrow[r]&  B_{0}\arrow[r] & 0.\\
	\end{tikzcd} \]
	
	We endow $A_{0}$ with the subspace topology from $A$, we endow $B_{0}$ with the subspace topology from $B$, and we endow $C_{0}$ with the tensor product topology -- this way all three topologies are $\varpi$-adic. We endow $I_{0}$ with subspace topology from $C_{0}$. In fact, since $B_{0}$ is $\varpi$-torsion free, the topology on $I_{0}$ is $\varpi$-adic, too. We endow $I_{0}/I_{0}[\varpi^{\infty}] $ and $ C_{0}/C_{0}[\varpi^{\infty}]$ with quotient topologies; they also are $\varpi$-adic. 
	
	Since all modules in the latest short exact sequence are $\varpi$-torsion free, the vertical arrows in the diagram
	\[ \begin{tikzcd}
		0\arrow[r]& I \arrow[r]&  C\arrow[r]&  B \arrow[r]& 0\\
		0\arrow[r]& (I_{0}/I_{0}[\varpi^{\infty}])[\dfrac{1}{\varpi}]\arrow[r]\arrow[u, equal]&  (C_{0}/C_{0}[\varpi^{\infty}])[\dfrac{1}{\varpi}] \arrow[r]\arrow[u, equal]&  B_{0}\arrow[r][\dfrac{1}{\varpi}]\arrow[u, equal]& 0\\
		0\arrow[r]& I_{0}/I_{0}[\varpi^{\infty}]\arrow[r]\arrow[u, hook]&  C_{0}/C_{0}[\varpi^{\infty}] \arrow[r]\arrow[u, hook]&\arrow[u]  B_{0}\arrow[r]\arrow[u, hook] & 0\\
	\end{tikzcd} \]
	are injective. If we topologize the modules in the middle row by declaring the modules in the bottom row to be open, we will recover the topologies of $I, C, B$ that they already possess\footnote{for the case of $C = B\otimes_{A}B$, see \cite[Proposition 1.2.2]{Hu96}. }. 
	
	The point of the above computation is that we would like to apply completion functor to the top row of the above diagram, so that we can compare the result with equation (\ref{equation}). The tricky part is to agrue that in this special case completion will preserve exactness. The argument is made by utilizing the three lemmas below.
	
	In what follows, localizations of a topological module will always be endowed with topology as above, i.e., by declaring the localization map to be open. In what follows, completions are \emph{always} meant to be with respect to the topology that the underlying topological abelian group carries, for example, the completion of a $\varpi$-adic module is exactly the $\varpi$-adic completion.
	
	\begin{lemma}[Slogan: when completion commutes with localization]\label{completion of localization = localization of completion}
		Let $R$ be a $\varpi$-adic ring, where $\varpi\in R$ is an element. Let $M$ be a topological $R$-module that is $\varpi$-torsion free and whose topology is $\varpi$-adic. 
		There are isomorphisms of topological $R[\frac{1}{\varpi}]$-modules \[  \widehat{\widehat{M}[\dfrac{1}{\varpi}]}\cong\widehat{M[\dfrac{1}{\varpi}]}\cong\widehat{M}[\dfrac{1}{\varpi}].\]
	\end{lemma}
	
	\begin{remark}
		While the first isomorphism holds even without the assumption of $M$ being $\varpi$-torsion free, this assumption is important for the second isomorphism to hold. For example, consider $R = \mathbb{Z}$, $\varpi = p$, and $M = \bigoplus_{n>0}\mathbb{Z}/p^{n}\mathbb{Z}$. Then, since $M=M[p^{\infty}]$, the module $\widehat{M[\frac{1}{\varpi}]}$ is zero. At the same time, $\widehat{M}[\frac{1}{\varpi}] \neq 0$, because $\widehat{M}$ has an element which is not killed by any power of $p$. To find such an element is left as an exercise to the reader. 
		
		To say even more, if $M$ is not $\varpi$-torsion free, we are afraid that the module $\widehat{M}$ could stop being complete after we invert $\varpi$, because modding-out $\varpi$-power torsion might spoil Hausdorfness. In case $M$ is $\varpi$-torsion free, the map $\widehat{M}\to \widehat{M}[\frac{1}{\varpi}] $ is injective (topological open immersion), and since $\widehat{M}$ is (Hausdorff) complete, the module $ \widehat{M}[\frac{1}{\varpi}] $ will remain (Hausdorff) complete, which is the crucial point in the proof below.
	\end{remark}
	
	\begin{proof}
		Once we show that the module $\widehat{M}[\frac{1}{\varpi}] $ is complete, it is not hard to check that each of the three topological $R[\frac{1}{\varpi}]$-modules from above (say, $\widehat{M}[\frac{1}{\varpi}]$) satisfies the following universal property. For any complete topological $R[\frac{1}{\varpi}]$-module $N$ and any (continuous) map of topological $R$-modules $f\colon M\to N$ there exists a unique map of topological $R[\frac{1}{\varpi}]$-modules $\tilde{f}\colon\widehat{M}[\frac{1}{\varpi}]\to N$ such that the diagram
		\[ \begin{tikzcd}
			&\widehat{M}[\dfrac{1}{\varpi}]\arrow[d, "\tilde{f}"]\\
			M\arrow[ur, "\text{canonical}"]\arrow[r, "f"]&N
		\end{tikzcd} \] commutes. The universal property would then yield the isomorphisms. It is therefore left to show that $\widehat{M}[\frac{1}{\varpi}]$ is already (Hausdorff) complete. This is clear from the remark above plus the following fact: if an $R$-module $M$ is $\varpi$-torsion free, then so is its completion $\widehat{M}$. We prove this fact now, so assume $M$ is $\varpi$-torsion free. Let $(m_{n})_{n>0}\in \varprojlim_{n> 0 }M/\varpi^{n}M$ be an element that satisfies\[  \varpi(m_{n})_{n>0} = (\varpi m_{n})_{n>0} = 0\in\varprojlim_{n> 0 }M/\varpi^{n}M. \]
		It means that $\varpi m_{n} = 0 \in M/\varpi^{n}M$. Because multiplication by $\varpi\colon M\to M$ is injective, we can deduce $m_{n}\in \varpi^{n-1}M$. Since $m_{n-1}$ is exactly the reduction of $m_{n}$ modulo $\varpi^{n-1}M$, we obtain $m_{n-1} = 0$. But $n$ was arbitrary, so $\widehat{M}$ is $\varpi$-torsion free, as desired.
	\end{proof}
	
	Lemma \ref{completion of localization = localization of completion} allows us to swap the order of localization and completion, and since localization is always exact, it suffices to show that $\varpi$-adic completion of the sequence 
	\[ \begin{tikzcd}
		0\arrow[r]& I_{0}/I_{0}[\varpi^{\infty}]\arrow[r]&  C_{0}/C_{0}[\varpi^{\infty}] \arrow[r]&  B_{0}\arrow[r] & 0\\
	\end{tikzcd} \] remains exact. Note, however, that all modules in this sequence are $\varpi$-torsion free. The combination of the two lemmas below yields what we want.
	
	\begin{lemma}[{\cite[Lecture III, Lemma 2.4]{Bh}}]
		Let $R$ be a ring, and let $\varpi\in R$ be an element. If an $R$-module $M$ has bounded $\varpi$-power torsion, then derived $\varpi$-completion of $M$ is concentrated in degree $0$ and coincides with its classical $\varpi$-completion.
	\end{lemma}
	
	Since $\varpi$-torsion free implies bounded $\varpi$-torsion, and since derived $\varpi$-completion is right exact in general, we immediatly get the right exact sequence
	\[ \begin{tikzcd}
		\widehat{I_{0}/I_{0}[\varpi^{\infty}]\arrow[r]}&  \widehat{C_{0}/C_{0}[\varpi^{\infty}] \arrow[r]}&  \widehat{B_{0}\arrow[r]}=B_{0} & 0.\\
	\end{tikzcd} \]
	
	We finish off with this lemma.
	\begin{lemma}
		Let $R$ be a ring, let $\varpi\in R$ be an element, and let $0\to L \overset{i}{\to}  M\overset{g}{\to} N\to 0$ be a short exact sequence of $R$-modules. If $N$ is $\varpi$-torsion free, then the induced map of $\varpi$-adic completions $\widehat{L}\overset{\widehat{i}}{\to}\widehat{M}$ is injective.
	\end{lemma}
	\begin{proof}
		Since the map \[ \widehat{i}\colon \varprojlim_{n> 0 }L/\varpi^{n}L \to \varprojlim_{n> 0 }M/\varpi^{n}M \] is given coordinate-wise by $i_{n}\colon L/\varpi^{n}L\to M/\varpi^{n}M$, it is enough to show that each individual $i_{n}$ is injective. Assume $l+\varpi^{n}L\in L/\varpi^{n}L$ maps to $0$, i.e., \[ i_{n}(l+\varpi^{n}L) = i(l)+\varpi^{n}M = l+\varpi^{n}M = 0\in M/\varpi^{n}M. \]
		This means that we can write $l = m\varpi^{n}$ for $m\in M$. By exactness of the sequence, $g(m\varpi^{n}) = g(m)\varpi^{n} = 0$. By assumption that $N$ is $\varpi$-torsion free, $g(m) = 0$. By exactness of the sequence, $m\in L$. Therefore, $l+\varpi^{n}L = 0 \in   L/\varpi^{n}L$, as desired.
	\end{proof}
	
	At last, we obtain the exact sequences
	
	\[ \begin{tikzcd}
		0\arrow[r]&\widehat{I_{0}/I_{0}[\varpi^{\infty}]\arrow[r]}&  \widehat{C_{0}/C_{0}[\varpi^{\infty}] \arrow[r]}&  B_{0} \arrow[r]& 0\\
	\end{tikzcd} \] and
	\[ \begin{tikzcd}
		0\arrow[r]&\widehat{I_{0}/I_{0}[\varpi^{\infty}]\arrow[r]}[\dfrac{1}{\varpi}]&  \widehat{C_{0}/C_{0}[\varpi^{\infty}] \arrow[r]}[\dfrac{1}{\varpi}]&  B_{0}[\dfrac{1}{\varpi}] \arrow[r]& 0,\\
	\end{tikzcd} \] 
	and applying Lemma \ref{completion of localization = localization of completion} we realize that the last sequence is just \[ 	0\to \widehat{I}\to \widehat{B\otimes_{A}B} \to B\to 0. \] Comparing with (\ref{equation}), we see that $J = \widehat{I}$. This finishes the first step of the computation.
	
	We now intend to deal with the second step of the computation, simultaneously proving Proposition \ref{description of the sheaf of differentials}. We will be able to state the second step after a couple of paragraphs. In any case, we begin the proof of Proposition \ref{description of the sheaf of differentials}.
	
	\begin{proof}
		By definition, $\Omega^{1}_{\spa{B}/\spa{A}} \cong J/J^{2}$, and we have just computed that this is the same as $\widehat{I}/\widehat{I}^{2} $. 
		
		On the one hand, we know that $\Omega^{1}_{\spa{B}/\spa{A}}$ is a vector bundle, so it corresponds to a finite projective $B$-module, in particular, it is complete. Hence, we can write $\widehat{I}/\widehat{I}^{2} = \widehat{\widehat{I}/\widehat{I}^{2}} $. 
		
		\begin{remark}
			We know \[ \widehat{\Omega^{1}_{B/A}} = \widehat{\Omega^{1}_{B_{0}/A_{0}}[\dfrac{1}{\varpi}]}, \] where $\Omega^{1}_{B_{0}/A_{0}}$ has its standard structure of a topological $B_{0}$-module\footnote{in fact, the topology is $\varpi$-adic in this case.}. We thus let the reader know how to topologize $\Omega^{1}_{B/A}$. 
		\end{remark}
		
		On the other hand,  \[ \widehat{\Omega^{1}_{B/A}} = \widehat{I/I^{2}}, \] and if we endow $\Omega^{1}_{B/A}$ with the quotient topology of $I/I^{2}$, we recover the topology from the above remark. 
		
		Therefore, we are left to show the equality of complete topological $B$-modules \[ \widehat{\widehat{I}/\widehat{I}^{2}}\cong\widehat{I/I^{2}}, \] proving Proposition \ref{description of the sheaf of differentials}. And indeed, this is the second step of our computation. Now, we carry it out.
		
		 Note that for $L = \widehat{\widehat{I}/\widehat{I}^{2}}$ or $L = \widehat{I/I^{2}} $ we have the canonincal map of topological $C$-modules $I\to L$. We will show that $L$ satisfies the following universal property: $L$ is a complete topological $B$-module together with a map of topological $C$-modules $I\to L$, such that for any other complete topological $B$-module $N$ and any map of topological $C$-modules $f\colon I\to N$ there exists unique map of topological $B$-modules $\tilde{f}\colon L\to N$ such that the diagram 
		\[ \begin{tikzcd}
			&L\arrow[d, "\tilde{f}"]\\
			I\arrow[ur, "\text{canonical}"]\arrow[r, "f"]&N
		\end{tikzcd} \]
		commutes. One should think about $L$ as of the completed tensor product $``\widehat{I\otimes_{C}B}"$, so that unraveling the universal properties of the completion and of the tensor product one by one, we get the stated universal property. The only problem with this way of thinking is that we do not know, a priori, how to topologize this tensor product (of modules, not of rings!), but in this special case we can endow $I\otimes_{C}B = I/I^{2}$ with quotient topology, and it saves the day for us. 
		
		We now prove that the universal property is satisfied for both $L = \widehat{\widehat{I}/\widehat{I}^{2}}$ and $L = \widehat{I/I^{2}} $, which would yield the desired isomorphism. We start with $L = \widehat{I/I^{2}} $. Consider the diagram
		\[ \begin{tikzcd}
			&\widehat{I/I^{2}}\arrow[d, "\tilde{f}"]\\
			I\arrow[ur, "\text{canonical}"]\arrow[r, "f"]&N.
		\end{tikzcd} \]
		Since $N$ is complete, by the universal property of completions, there is a bijection between $\tilde{f}$ and $\overline{f}$:
		\[ \begin{tikzcd}
			&{I/I^{2}}\arrow[d, "\overline{f}"]\\
			I\arrow[ur, "\text{canonical}"]\arrow[r, "f"]&N.
		\end{tikzcd} \]
	Now, by the tensor product adjunction, a map of $B$-modules $\overline{f}\colon I/I^{2}\to N$ is the same as a map of $C$-modules $f\colon I\to N$ (note that $I/I^{2} = I\otimes_{C}B$). That continuous $\overline{f}$ correspond to continuous $f$ follows from the fact that the map $I\to I/I^{2}$ is a topological quotient map. That deals with $L = \widehat{I/I^{2}} $.
	
	We show the same universal property for $L = \widehat{\widehat{I}/\widehat{I}^{2}}$. Consider the diagram 
	\[ \begin{tikzcd}
		&&\widehat{\widehat{I}/\widehat{I}^{2}}\arrow[d, "\tilde{f}"]\\
		I\arrow[r, "\iota"]\arrow[rr, bend right, "f"]&\widehat{I}\arrow[ur, "\text{can.}"]\arrow[r, "\widehat{f}"]&N.
	\end{tikzcd} \]
	By the univeral property of completion, there is a bijection between $\tilde{f}$ and $\overline{f}$:
		\[ \begin{tikzcd}
		&&{\widehat{I}/\widehat{I}^{2}}\arrow[d, "\overline{f}"]\\
		I\arrow[r, "\iota"]\arrow[rr, bend right, "f"]&\widehat{I}\arrow[ur, "\text{can.}"]\arrow[r, "\widehat{f}"]&N.
	\end{tikzcd} \]
	Now, by the tensor product adjunction, a map of $B$-modules $\overline{f}\colon \widehat{I}/\widehat{I}^{2}\to N$ is the same as a map of $\widehat{C}$-modules $\widehat{f}\colon \widehat{I}\to N$ (note that $\widehat{I}/\widehat{I}^{2} = \widehat{I}\otimes_{\widehat{C}}B$). That continuous $\overline{f}$ correspond to continuous $\widehat{f}$ follows from the fact that the map $\widehat{I}\to \widehat{I}/\widehat{I}^{2}$ is a topological quotient map. Finally, maps of topological $\widehat{C}$-modules $\widehat{f}\colon \widehat{I}\to N$ are in bijection with maps of topological $C$-modules $f\colon I\to N$, by the universal proprety of completion. 
	
	\end{proof}

\end{document}